\documentclass[10pt]{amsart}
\pdfoutput=1
\usepackage{amsmath, amsthm, amssymb,slashed,stmaryrd}
\usepackage{ifpdf}
\usepackage[pdftex]{graphicx}
\usepackage{tikz}
\usetikzlibrary{matrix,arrows,calc}
\usepackage{mathtools}
\usepackage{todonotes}
\usepackage[mathscr]{eucal}

\usepackage[pdftex,plainpages=false,hypertexnames=false,pdfpagelabels]{hyperref}

 \usepackage[utf8]{inputenc}
\usepackage[T5,T1]{fontenc}
 \theoremstyle{plain}
 \newtheorem{thm}{Theorem}[section]
    \newtheorem{claim}{Claim}[section]
 \newtheorem{cor}[thm]{Corollary}
 \newtheorem{lem}[thm]{Lemma}
 \newtheorem{prop}[thm]{Proposition}
 \newtheorem{conj}[thm]{Conjecture}
 \theoremstyle{definition}
 \newtheorem{defn}[thm]{Definition}

 \newtheorem*{thm*}{Theorem}
 \theoremstyle{remark}
 \newtheorem{rmk}[thm]{Remark}
 \newtheorem{example}[thm]{Example}
 \newtheorem{question}[thm]{Question}
 \numberwithin{thm}{subsection} 
 
\def\beq{\begin{eqnarray}}
\def\eeq{\end{eqnarray}}
 \newcommand{\bp}{\begin{proof}[Proof]}
 \newcommand{\ep}{\end{proof}}

\DeclareSymbolFont{bbold}{U}{bbold}{m}{n}
\DeclareSymbolFontAlphabet{\mathbbold}{bbold}

\newcommand\nc{\newcommand}
\nc\mf\mathfrak
\nc\mc\mathcal
\nc\mb\mathbb
\nc\ms\mathscr

\def\Bun{{\rm Bun}}

\def\Sym{{\rm Sym}}

\def\Hom{{\sf Hom}}
\def\SL{{\rm SL}}

\def\Sh{{\rm Sh}}

\def\std{{\rm Std}}
\def\Cl{{\rm Cl}}
\def\K3{{\rm K3}}
\def\Imag{{\rm Im}} 
\def\Re{{\rm Re}}
\def\SO{{\rm SO}}
\def\herm{{\rm Herm}}
\def\GL{{\rm GL}}
\def\SL{{\rm SL}}
\def\SU{{\rm SU}}

\def\O{{\rm O}}
\def\Sp{{\rm Sp}}
\def\U{{\rm U}}
\def\GSp{{\rm GSp}}
\def\Ad{{\rm Ad}}
\def\Map{{\rm Map}}
\def\Tr{{\rm Tr}}
\def\Res{{\rm Res}}
\def\lambdahat{{\hat{\lambda}}}
\def\kappahat{{\hat{\kappa}}}
\def\Stab{{\rm Stab}}
\def\sgn{{\rm sgn}}

\newcommand{\defeq}{\vcentcolon=}

\usepackage[
backend=biber,
style=alphabetic,
sorting=nyt,
maxnames=50
]{biblatex}

\begin{document}

\title{The Attractor Conjecture for Calabi-Yau variations of Hodge structures}

\author{Yeuk Hay Joshua Lam}

\date{\today}

\begin{abstract} We study attractor points for Calabi-Yau variations of Hodge structures. In particular, for certain moduli spaces which are Shimura varieties,  we prove  that the attractor points  are CM points, thus proving  Moore's Attractor Conjecture  in these cases. We also study \emph{non-BPS} examples of attractors, obtaining special points on locally symmetric spaces without hermitian structures, as well as locally symmetric spaces inside Shimura varieties; for the latter, we  point out a possible analogy with locally symmetric subspaces studied by Goresky-Tai. Finally we give an explicit geometric  description of non-BPS attractor points in the simplest case.
\end{abstract}

\maketitle 
\setcounter{tocdepth}{1}
\tableofcontents

\section{Introduction}
The attractor mechanism is a procedure which picks out  special points in any moduli space of Calabi-Yau threefolds (CY3s), and originated in the study of black holes in string theory compactifications; this was discovered in the pioneering work of Ferrar-Kallosh \cite{ferrarakallosh}. Roughly speaking,  for $X$ a CY3 we consider its moduli space $\mc{M}$, and then for  each vector in $H^3(X, \mb{Z})$ we have a real valued potential function on (the universal cover of) $\mc{M}$ given by the formula
\[
X' \mapsto \frac{|\int_{\gamma}\Omega|^2}{\int_{X'} \Omega \wedge \overline{\Omega}},
\]
where $X'$ is the CY3 corresponding to any other point in $\mc{M}$, and $\Omega$ is a global choice of non-vanishing holomorphic 3-form.  The  attractor points are defined to be the critical points of this function. A calculation shows that  this is equivalent to the following Hodge theoretic condition: 

\begin{defn} A CY3  $X$ corresponding to $P\in \mc{M}$ is an attractor point for the class $\gamma\in H^3(X, \mb{Z})$ if  
\begin{equation}\label{eqn:hodgeattcond}
\gamma\in H^{3,0}\oplus H^{0,3},
\end{equation}
where we have written  
\[
H^n(X, \mb{C})=\bigoplus_{i+j=n}H^{ij}
\]
for the Hodge decomposition. 
\end{defn}
\begin{rmk}
Following the physics literature we will often refer to $\gamma$ as the \emph{charge vector}.
\end{rmk}
In his pioneering work on the attractor mechanism,  Moore conjectured in \cite{moore} that attractor points are in fact defined over $\overline{\mb{Q}}$, and verified it in several cases. This is known as the attractor conjecture.

In this paper we study the Attractor Conjecture in the context of Calabi-Yau variation of Hodge structures (CYVHS). We study a  family of such variations of Hodge structures, which  have been studied by many different groups of authors previously: they were initially defined by Gross \cite{gross} as $\mb{R}$-VHS,  extended by Sheng-Zuo \cite{shengzuo}, and were recently considered in detail again by Friedman-Laza \cite{friedmanlaza}. A special  feature of these examples is that the spaces parametrizing the Hodge structures are in fact Shimura varieties; in fact, all hermitian symmetric domains possess \emph{canonical} Calabi-Yau variation of Hodge structures, and we study all the weight three examples (which is most related to threefolds). Physically these arise in the \emph{very special} $4d \ N=2$ supergravity theories: we refer the reader to \cite{ferrarasymmetricspace} for a review. We will refer to these as Gross' VHS, though strictly speaking Gross only considered the $\mb{R}$-VHS, and there exist many distinct $\mb{Q}$-descents of each of these.

In fact, there is good reason to believe that the Attractor Conjecture holds if and only if the moduli space is a Shimura variety: the recent work \cite{medolgachev} of the author and A. Tripathy  exhibits counterexamples to the Attractor Conjecture, albeit in higher dimensions, showing precisely this dichotomy between Shimura and non-Shimura moduli. One of the goals of this paper is to give further evidence to this claim by showing that the Attractor Conjecture works perfectly when the moduli space is a Shimura variety. Our attractor points in fact satisfy the stronger property of being \emph{rank 2}, a distinction that was stressed by Moore from the beginning. 

 We will review these variations of Hodge structures in Section \ref{section:CYVHS}; they are labelled by the  so-called degree 3 Jordan algebras.  Since the attractor condition (\ref{eqn:hodgeattcond}) is purely a condition on the rational Hodge structure, we can define attractor points as in the case of Calabi-Yau threefolds. We can now state a rough version of the main result (for a precise statement see Theorem~\ref{thm:cmattractor}).

\begin{thm}\label{thm:main1}
For each of Gross' VHS, the attractor points are CM points. 
\end{thm}
\begin{rmk}
In particular, we get an explicit parametrization of (some) CM points in Shimura varieties of $E_7$ type: we are not aware of any such parametrization in the existing literature.
\end{rmk}
\begin{rmk}
The attractor points have an interesting relation to certain special functions studied by Pollack recently in his theory of Fourier expansions for quaternionic modular forms \cite{pollack}; for example, they are indexed by the same set, and the attractor points are where Pollack's functions achieve their ``peaks''. This will be explicated in a future work.
\end{rmk}
The proof of Theorem \ref{thm:main1} uses crucially the \emph{special geometry} discovered by Strominger \cite{stromingerspecial}, which are special coordinates on the moduli spaces appearing in the relevant  physical theories.

We make some brief remarks about the general paradigm of the attractor mechanism. The ingredients are the following:
\begin{itemize}
    \item A  space $\mc{M}$ (usually a moduli space of Calabi-Yau threefolds, or in our case the space parametrizing a VHS)
    \item a rational vector space $V$ (sometimes referred to as the charge lattice),
    \item a real valued potential function $V(p)$ on $\mc{M}$ for each element $\gamma \in V$.
\end{itemize}
For example, in the setting of Theorem \ref{thm:main1}, the space $\mc{M}$ will be the Shimura variety underlying the CYVHS, and $V$ will be a fiber of the VHS at some point (we may identify all the different fibers by parallel transport). Given data as above, the critical points of the function $V(p)$ are known as the attractor points with charge $p$.

In each of the examples considered in this paper there is an additional ingredient which means that the moduli space is a symmetric space:
\begin{itemize}
    \item there exists compatible actions of an algebraic group $G$ on $V$ and $\mc{M}$.
\end{itemize}
In the examples in Theorem \ref{thm:main1}, the pair $(G,V)$ is a prehomogeneous vector space, and so in summary we have special points in the locally symmetric space $G(\mb{Z})\backslash \mc{M}$ associated to each orbit in $V(\mb{Z})/G(\mb{Z})$.

\begin{rmk}
There is an intriguing construction by J.Thorne in \cite{thorne} (and also implicitly by Ho-Le Hung-Ngo in \cite{ngo}) associating points in a locally symmetric space to  integral orbit of some group $G$ acting on a vector space $V$ to study the average size of Selmer ranks of elliptic curves, but this time in the function field setting; in loc.cit the equidistribution of these points was proved and was used to count Selmer elements. The map in loc.cit. can be rewritten as 
\[
\Map(C, V/G)(\mb{F}_q)\rightarrow \Map(C, */G)(\mb{F}_q)=\Bun_G(\mb{F}_q).
\]
We would like to suggest that the attractor mechanism gives an arithmetic analogue of this map, using the analogy between $\Bun_G$ and Shimura varieties.
\end{rmk}


We now come to the second goal of this work, which is to investigate \emph{non-BPS} attractors. Strictly speaking the points in Theorem \ref{thm:main1} have the further property that they are the so-called  ``Bogomol'nyi-Prasad-Sommerfeld'' (BPS) attractor points. There is in fact a more general class known as ``extremal attractors'' which can be either BPS (these are the classical ones  in Theorem \ref{thm:main1}) or non-BPS. A second goal of this work is to  explore these non-BPS attractors; briefly, we obtain exotic analogues of ``CM points on Shimura varieties'' depending on the setup we consider, including locally symmetric subspaces on Shimura varieties, as well as special points on locally symmetric spaces (the latter have no  hermitian  structures). We hope these will be of arithmetic interest, and we now conclude the introduction by giving some details on them. 


\begin{thm}\label{thm:intrononbps}\hfill
\begin{enumerate}
\item 
In the $d=4$, \emph{non-BPS} cases, for a fixed charge vector $\gamma$ the  attractors form a  locally symmetric subspace $\mc{N}$ of $\mc{M}$; 
\item in the $d=5$, \emph{BPS}  cases, the attractor points  give special points on locally symmetric spaces (with no hermitian structures) analogous to  Kudla-Millson cycles. 
\end{enumerate}
\end{thm}
Statements similar to these appear in various works in the physics literature, though we are not aware of any proofs. In any case it should be clear that the above result owes much to these works, and one of our hopes is to draw mathematicians' attention to these special cycles.



\begin{rmk}\label{rmk:goreskytai}
In (1) of Theorem \ref{thm:intrononbps} the spaces $\mc{N}$ have dimension  just less than half that of $\mc{M}$: more precisely, if $\mc{M}$ has complex dimension $n$, then $\mc{N}$ will have real dimension $n-1$. We believe these subspaces are analogues of those studied by Goresky-Tai in \cite{goreskytai}\cite{goreskytaimultiplication}, and should be thought of as the fixed points of some anti-holomorphic involution acting on $\mc{M}$. In fact, in the case when the Jordan algebra is $J=\herm_3(\mb{R})$ we have copies of the locally symmetric space for $\GL(3, \mb{R})$ embedded into the $\Sp(6, \mb{R})$ Shimura variety, which should be closely related to the $n=3$ case (following their notaion) in \cite{goreskytai}. We hope to return to this point in future work. 

In fact, the existence  of such locally symmetric spaces inside hermitian symmetric domains given by Jordan algebras seems to have been anticipated by the authors of loc.cit. already.
\end{rmk}
\begin{rmk}
For example, for (2) in Theorem \ref{thm:intrononbps} one of the cases furnishes us with  ``special points'' on the locally symmetric space whose real group is $E_{6(-26)}$. Kudla-Millson \cite{kudlamillson} introduced  special cycles for  orthogonal and    unitary groups in order to study their intersection theory and relation to automorphic forms. It is therefore natural to wonder if there is some Kudla-Millson theory for the exceptional cases mentioned above. 
\end{rmk}

We can also give an explicit geometric description of the non-BPS attractor points in the simplest case, which we now describe. In this case, the moduli space is simply the upper half plane $\mb{H}$, and the  variation of Hodge structures is $V=\textrm{Sym}^3V_{std}$, where $V_{std}$ is the weight 1 Hodge structure associated to the universal elliptic curve over $\mc{H}$. We may equivalently think of a fiber $V_x$ as the space of binary cubic forms, and the distinction between BPS and non-BPS stems from whether the discriminant this binary cubic form is positive (BPS) or negative (non-BPS); more precisely, the attractors associated to the form $F=aX^3+bX^2Y+cXY^2+dY^3\in V$  are BPS if the polynomial $f=ax^3+bx^2+cx+d$ has three real roots, and non-BPS if $f$ has only one real root and a pair of complex conjugate complex roots. Note that by the remark \ref{rmk:goreskytai} the non-BPS attractors are also just points in this case, since the complex dimension of $\mb{H}$ is $n=1$, and therefore the non-BPS attractor points have moduli space of dimension $n-1=0$.

\begin{thm}\label{thm:main3}
The attractor points associated to $F$ are given concretely as follows. Embed $\mb{H}$ into $\mb{H}_3$, the hyperbolic upper half space in the natural way, i.e. taking the model of $\mb{H}_3$ as the interior of $\mb{P}^1(\mb{C})$, and $\mb{H}$ as the interior of $\mb{P}^1(\mb{R})$, then $\mb{P}^1(\mb{R}) = \partial \mb{H} \subset \partial \mb{H}_3=\mb{P}^1(\mb{C})$ is the equator circle of the sphere. If $\alpha, \beta, \gamma\in \mb{P}^1(\mb{C})$ are the roots of $F$ (with positive or negeative discriminant), then the attractor point is the center of mass of $\alpha, \beta, \gamma$ inside $\mb{H}_3$.
\end{thm}

Note that this theorem also applies to the case of binary cubic forms $f$ with positive discriminant. In that case it is a classical fact that the center of mass of the three (real roots) is a CM point, and so this gives a proof of the attractor conjecture in this case.

Finally we address the question of equidistribution of attractor points. This has been studied by Douglas-Shiffman-Zelditch in a series of works \cite{DSZ1, DSZII, DSZIII}. For the case of the non-BPS attractors in Theorem \ref{thm:main3}, we write down, following \cite{kachruetc}, an explicit density function to which the distribution of attractor points limit to. This is the analogue of a result of Hough \cite{hough}, who proved equidistribution of the 3-torsion points of class groups of imaginary quadratic fields, which are the BPS attractor points in this example.

\begin{conj}[See Section 4.2]
The non-BPS attractor points in Theorem \ref{thm:main3} equidistribute in  a certain explicit distribution $\rho$.
\end{conj}

For the formula for $\rho$ we refer the reader to Section \ref{subsection:t^3distribution}. Note that this $\rho$ does not seem to be the uniform distribution on $\mb{H}$.

This paper is organized as follows: Section \ref{section:specialgeometry} gives a review on special geometry, to be used in the proof of Theorem \ref{thm:main1}. Section \ref{section:CYVHS} is a detailed study of the  Jordan algebras, the relevant hermitian symmetric domains and Shimura varieties, as well as the CYVHS on them, and concludes with the proof that attractor points are CM. Section \ref{section:nonbps} studies the simplest case of non-BPS attractors, and gives the geometric description of the attractor points in this case. In sections \ref{section:nonbpsgeneral} and \ref{section:otheravatars} we study non-BPS attractors in general, and end by giving some evidence that these are related to Goresky-Tai's \emph{moduli of real abelian threefolds}. Finally we conclude in Section \ref{section:further} by stating some questions and further directions.



\section{Review on special geometry}\label{section:specialgeometry}
We review some geometry relevant to moduli spaces of $4d, \ \mc{N}=2$ theories, which will be essential in the proof of our result on CM attractor points.
\begin{defn} A K\"ahler manifold $\mc{M}$ of complex dimension $n$ is called special K\"ahler if it is equipped with a flat holomorphic vector bundle of dimension $2n+2$ with real structure and symplectic form $\langle \cdot , \cdot \rangle$ (in other words, a $\Sp(2n+2, \mb{R})$-structure), and locally a section $v$ such that the function
\[
K=-\log (-i\langle v, \bar{v}\rangle)
\]
is a K\"ahler potential, and furthermore $\langle v, \partial_i v\rangle =0$ for all $i$

\end{defn}
\begin{rmk}
An example to keep in mind is the moduli space of a Calabi-Yau threefold, in which case the holomorphic vector bundle $\mc{E}$ is given by $H^3(X, \mb{C})$, and the section $v$ by  the holomorphic 3-form $\Omega \in H^3(X, \mb{C})$. Note that the condition $\langle \Omega, \partial_i\Omega\rangle $ holds  since $\partial_i\Omega \in h^{2,1}$ by Griffiths transversality. This was one of the initial motivations for introducing the notion of a special geometry.
\end{rmk}

Note that we have  an embedding of the universal cover $\tilde{\mc{M}}$ into the vector space $\mb{C}^{2n+2}$:
\begin{align*}
\widetilde{\mc{M}}&\rightarrow  \mb{C}^{2n+2}\\
p&\mapsto v_p,
\end{align*}
where $v_p$ denotes the section $v$ at the point $p$, and we have used the flat structure to trivialize the bundle.

Writing the above map as $(X^I, F_I)$ where $I=0, 1, \cdots , n$ gives local homogeneous coordinates $X^I$ on $\mc{M}$.  In all the cases we will study in what follows there will also be the following piece of data:

\begin{defn}
A homogeneous function $\mc{F}(X^I)$ of degree 2 is called a prepotential if we have 
\[
F_I=\frac{\partial \mc{F}}{\partial X^I}.
\]
\end{defn}
\begin{example}
The simplest example is the case of $\tilde{\mc{M}}$ being the upper half plane $\mb{H}$, with homogeneous coordinates $(X^0, X^1)$ where $\tau=X^1/X^0$ is the usual coordinate on $\mb{H}$, and the prepotential is given by
\[
\mc{F}=\frac{(X^1)^3}{X^0}.
\]

\end{example}

\subsection{Real Darboux coordinates}\label{section:realdarboux}
In this section we review an alternative set of (real) coordinates on $\mc{M}$, which will be useful later on in the proof that attractor points are CM.  These were introduced by  Cecotti-Ferrara-Girardello  \cite{cecottigeometry} and by Hitchin \cite{hitchin} independently. Recall that in the last section we introduced holomorphic homoegeneous coordinates 
\[
X^I: \ I=0, \cdots, n,
\]
and also certain ''dual'' coordinates
\[
F_I=\frac{\partial F}{\partial X^I}: \ I=0, \cdots, n.
\]
Now let us define real functions $p^I, q_I, \phi^I, \psi_I$ by taking real and imaginary parts of these:
\begin{align}
\begin{split}
    X^I&=p^I+i\phi^I \\
    F_I&=q_I+i\psi_I.
    \end{split}
\end{align}
Note that the choice of $p^I, q_I$, coinciding with our notation for the charge vector $Q=(p^I, q_I)$ later on, will become clear when we discuss the attractor equations: we hope this will not cause any confusion. 

The $p^I, q_I$ for $I=0, \cdots, n$ give real homogeneous coordinates on the moduli space. The  following  expresses the  functions $\phi^I=\Imag X^I , \psi_I=\Imag F_I$ in terms of the $p^I, q_I$'s. This result is due to Cecotti-Ferrara-Girardello \cite{cecottigeometry} and Hitchin \cite{hitchin} independently, though Hitchin works in the setting of the special geometry underlying the moduli of homolomorphic lagrangians inside a holomorphic symplectic variety. We refer the reader to \cite[Proposition 1.24]{freed} as well as \cite[Equations (27, 28, 29)]{ferraraobservations} for references.
\begin{lem}\label{lemma:hamiltonian}
There is a real valued function $S(p,q)$ such that 
\begin{align}
\begin{split}
\phi^I &=\frac{\partial S(p,q)}{\partial q_I},\\
\psi_I &=-\frac{\partial S(p,q)}{\partial p^I}.
\end{split}
\end{align}
Furthermore, $S$ is the Legendre transform of the imaginary part of the prepotential $F(X^I)$; more precisely, $S$ is given by the formula 
\[
S(p,q)=\sum_Iq_I\phi^I(p,q)-\Imag F (p, \phi(p,q)).
\]
In the above expression we view, by a slight abuse of notation, $\Imag F$ as a function of $p^I, \phi^I$ since the $X^I$'s a re functions of the $p^I, \phi^I$'s and $F$ is a function of the $X^I$'s.

\end{lem}
\section{Calabi-Yau variation of Hodge structures}\label{section:CYVHS}
The motivation for the attractor conjectures came from string theory and in particular its low energy limit which is a supergravity (SUGRA) theory: in other words, for every Calabi-Yau threefold there is an associated supergravity theory, and the study of spherically symmetric black holes in this gravity theory gives rise to the existence of special points in moduli space. It turns out that there is a class of supergravity theories (known as very special SUGRAs), for which the moduli spaces are symmetric spaces, and we can still study the attractor points. In these cases we obtain precisely the \emph{canonical} variations of Hodge structures on certain Shimura varieties. There is a more general class of theories for which the moduli space is \emph{homogeneous} but not symmetric; we make some brief remarks on these in Section~\ref{subsection:talg} and intend to study them in the future.

\begin{defn}\label{defn:CYVHS}
A rational pure Hodge structure of weight $n$ is said to be of Calabi-Yau type if $h^{n,0}=1$.
\end{defn}

Although for many of these models  we do not know that the variation of Hodge structure is realized geometrically by a family of CY3's, this is conjectured to be the case: in the mathematics literature this is a question of B.Gross \cite{gross}, and in the physics literature this bears the name of the Gepner conjecture \cite{gepner}. For a review of these moduli spaces and the corresponding variation of Hodge structure see \cite{heateqn}, and below we follow  closely the exposition in loc.cit.. 

\subsection{Summary of Calabi-Yau VHS over hermitian symmetric domain}
We record here some existing results on canonical variation of Hodge structures of CY type. In \cite{gross} Gross constructed real variation of Calabi-Yau Hodge structures on all hermitian symmetric tube domains. This was extended by Sheng-Zuo \cite{shengzuo} to all bounded hermitian symmetric domains. More precisely they made the following definition.

\begin{defn}
For a (pure) complex PVHS of weight $n$ over a base complex manifold $S$, we say it is of Calabi-Yau type of Type 1 if 
\begin{itemize}
\item $V$ has a real structure;
    \item $\dim_{\mb{C}} V^{n,0}=1$;
    \item the Kodaira-Spencer map 
    \[
    \theta: \mc{T}\rightarrow \Hom(V^{n,0}, V^{n-1,1})
    \]
is an isomorphism (here $\mc{T}$ dneotes the tangent bundle of $S$).
\end{itemize}
If only the second and third conditions are satisfies, that is, if $V$  does not possess an $\mb{R}$-structure, then we say that $V$ is  a VHS of CY type of Type II.
\end{defn}

The following was then proved by Sheng-Zuo in \cite{shengzuo}:
\begin{thm}
For each irreducible Hermitian symmetric domain there exists a $\mb{C}$-PVHS of Calabi-Yau type. Furthermore it is of Type I in the case of tube domains, and the $\mb{R}$-PVHS agrees with the one constructed by Gross, and it is of Type II otherwise.
\end{thm}

Now since the attractor conjecture only really makes sense if the VHS has a real structure, for the complex cases constructed by Sheng-Zuo the best that one can do is to force it to have a real structure by replacing $V$ by $V\oplus \overline{V}$, where $\overline{V}$ denotes the complex conjugate of $V$. The $\mb{R}$-VHS of CY type were then classified by Friedman-Laza; more precisely they proved the following theorem:

\begin{thm}
For each irreducible Hermitian symmetric domain $D$ there is a unique minimal $\mb{R}$-CYVHS over $D$ of minimal weight. For the tube domains these are  precisely the ones constructed by Gross, and for the other cases they are given by $V\oplus \overline{V}$ where $V$ is the $\mb{C}$-VHS constructed by Sheng-Zuo.
\end{thm}

Since the Attractor Conjecture requires a $\mb{Q}$-structure, we will be interested in  $\mb{Q}$-VHS. We note that for the complex cases we have the following straightforward proposition.  

\begin{prop}
For a $\mb{Q}$-CYVHS $V$ of Type II, i.e. for which $V_{\mb{R}}=W\oplus \overline{W}$ for a $\mb{C}$-VHS $W$ of   type II, the attractor conjecture holds.
\end{prop}
\begin{proof}
In this case, the Hodge structure $V=W\oplus \bar{W}$ has weak CM by an imaginary quadratic field $E$, and we have the Hodge decomposition
\[
W\otimes_{E}\mb{C}=H^{3,0}\oplus h^{2,1}.
\]
The attractor condition amounts to the statement that that the line $H^{3,0}\subset W\otimes_E\mb{C}$ is  defined over $E$, and thus $H^{3,0}\oplus H^{0,3} \subset V\otimes_{\mb{Q}} \mb{C}$ is a rational sub Hodge structure, and so by Proposition \ref{prop:rigid} we are done.
\end{proof}

In particular, for the Shimura varieties  of  $E_6$ type (which requires an auxiliary choice of an imaginary quadratic field $E$), we have a parametrization of certain CM points. As with the case of the $E_7$ type Shimura variety it would be interesting to find a geometric realization for this weight 3  CYVHS.

\subsection{Very special SUGRAs from Jordan algebras}\label{subsection:jordan}

In this section we review the properties of the very special sugras coming from Jordan algebras. The point is that starting from the data of a Jordan algebra we will construct a symmetric space with a Calabi-Yau variation of Hodge structures on it of weight 3. 

\subsubsection{Recollections on Jordan algebras} 
In this section we fix some notations for Jordan algebras, their cubic norm structures, and various groups assocaited with them.


Since most of the information carried by our Jordan algebras is encoded in its ''cubic norm structure'', we now recall this notion (we follow the exposition given in \cite[Section 2.1]{pollack}).

\begin{defn}
A cubic norm structure on a vector space $J$ over a field $F$ (of characteristic zero) is the data of 
\begin{itemize}
\item an element $1_J\in J$,
    \item a cubic polynomial $N: J\rightarrow F$,
    \item  a quadratic polynomial map (the \textit{sharp} map)$\#: J\rightarrow J$,
    \item and a symmetric bilinear pairing (called the trace form) $(\cdot ,\cdot ):J\otimes J \rightarrow F$.
\end{itemize}
These are required to satisfy conditions (1)-(4) below, which we now describe. For $x, y \in J$ we set 
\[
x\times y =(x+y)^\#-x^\#-y^\#
\]
and 
\[
(\cdot, \cdot, \cdot): J\otimes J \otimes J\rightarrow F
\]
to be the (unique) trilinear form such that $(x,x,x)=6N(x)$. The data specifided above are required to satisfy the following conditions
\begin{enumerate}
    \item $N(1_J)=1, \ 1_J^\#=1_J$ and for all $x\in J$ we have $1_J\times x=(1_J, x)-x$;
    \item $(x^\#)^\#=N(x)x$ for all $x\in J$;
    \item for all $x,y \in J$ we have $(x,y)=(1_J, 1_J, x)(1_J, 1_J, y)-(1_J, x, y)$;
    \item for all $x,y \in J$, we have $N(x+y)=N(x)+(x^\#, y)+(x, y^\#)+N(y)$.
\end{enumerate}
\end{defn}
We will often refer to a cubic norm structure on $J$ simply as $N$, and understand that we also have the additional pieces of data $\#$ and  $(\cdot, \cdot)$. Note that in all the examples we consider, the cubic norm structure will be non-degenerate in the sense that the trace form $(x,y)$ is non-degenerate. In this case the $\#$ map is specified by $N$ by the formula
\[
(x^{\#}, y)=\frac{(x,x,y)}{2}.
\]

\begin{rmk} From a  cubic norm structure one can define a Jordan algebra on the same vector space $J$ via the so-called Springer construction; roughly, a Jordan algebra $J$ is a vector space equipped with a \textit{Jordan product} $\circ : J\times J\rightarrow  J$ satisfying 
\begin{align*}
    x\circ y&=y\circ x\\
    (x\circ y)\circ x^2&=x\circ (y \circ(x^2));
\end{align*}
here we denote by $x^2$ the self product $x\circ x$ of an element $x\in J$.

We now give a brief description of the Jordan algebra constructed from a cubic norm structure. The vector space underlying the Jordan algebra is just the same as that of the cubic norm structure, and the Jordan product is given by 

\[
x\circ y=\frac{1}{2}(x\times y+\Tr (x)y+\Tr(y)x-S(x,y)1_J),
\]
where 
\begin{equation*}
S(x,y)=(x,y, i_J),\ \Tr(x)=\frac{(1_J, 1_J, x)}{2}.
\end{equation*}
Since we will not need this additional piece of structure explicitly, we simply refer the reader to  \cite[Section 2.3]{krutelevich} for the construction of the Jordan product.
\end{rmk}
\begin{rmk}
We will sometimes also refer to $M_J$ as $G_5$ and $H_J$ as $G_4$, since these are the symmetry groups of the physical theory in $5$ and $4$ dimensions, respectively.
\end{rmk}
Let us  start with a degree three Jordan algebra $J$ over $\mb{Q}$. The classification of such Jordan algebras implies that  the real points of $J$ is given by one of the following:
\[
\mb{R}\oplus \Gamma_{n-1,1}, \ \mb{R},\  \herm(\mb{R}),\  \herm_3(\mb{C}), \ \herm_3(\mb{H}),\ \herm_3(\mb{O}),
\]
where $\Gamma_{n-1,1}$ denotes a degree 2 Jordan algebra with signature $(n-1,1)$ (see Example (\ref{example:cubicnormgeneric}) below), and $\mb{R},\mb{C}, \mb{H}, \mb{O}$ denote the reals, complex numbers, quaternions, and octonoions respectively, and $Herm_3(\mb{K})$ denotes $3\times 3$ hermitian matrices with coefficients in $\mb{K}$. The algebra $J$ comes equipped with a cubic norm $\mc{V}$ which is essentially the determinant; see Example \ref{example:cubicnorm} for details of this norm. 

\begin{example}\label{example:cubicnormgeneric}
We give the example of cubic norm structures defined by quadratic spaces. Suppose we have  a vector space $Q$ over a field $F$ equipped with a non-degenerate quadratic form $B_0: Q\rightarrow F$, as well as an element $c_0\in Q$ such that $Q(c_0)=1)$. We now  define a  cubic norm structure on $V\defeq F\oplus Q$. The cubic norm   is defined simply by  
\begin{align*}
N(\alpha, x)&=\alpha B(x) \ \text{for} (\alpha, x)\in V,\\
\end{align*}
and the $\#$ map and pairing $(\cdot, \cdot) $ are given explicitly by 
\begin{align}
    x^{\#}=(B(x_0), \alpha x_0^*)
    (x,y)&=\alpha \beta +B(x_0*, y_0),
\end{align}
for $x, y\in V$. Here we have written $x=(\alpha, x_0), y=(\beta, y_0$, and $x_0^*=B(x_0, c_0)c_0-x_0$

\end{example}

\begin{example}\label{example:cubicnorm}
We introduce the exceptional examples which are considered in this paper. In this case we fix a descent of $\mb{R}, \ \mb{C}, \ \mb{H}, \ \mb{O}$ over $\mb{Q}$ (that is, $\mb{Q}$ itself, an imaginary quadratic field, a quaternion algebra, or an octonion algebra, respectively). By definition the algebras $\mb{K}$ are equipped with  conjugation maps $\bar{\cdot}$, as well as  norm and trace maps 
\begin{align*}
|\cdot|&: \mb{K}\rightarrow \mb{Q}_{\geq 0}\\
\Tr&: \mb{K}\rightarrow \mb{Q}
\end{align*}
which are defined by
\[
|k|^2=k\bar{k}, \ \Tr(k)=k+\bar{k}
\]
for $k\in \mb{K}$. Note that (even for the octonions $\mb{O}$) we have
\[
\Tr((xy)z)=\Tr(x(yz)),
\]
and we denote this quantity by $\Tr(xyz)$. 

 Then  the Jordan algebra has as its underlying vector space the $3\times 3$ hermitian matrices with entries in $\mb{K}$:
 \[
 J=\herm_3(\mb{K}).
 \]
 A typical element $x\in J$ will be denoted as
 \[
x=\begin{pmatrix}  a & z & \bar{y} \\
                     \bar{z} & b &x\\
                    y & \bar{x} & c
\end{pmatrix};
\]
$a,b,c\in \mb{R}, \ x,y,z \in \mb{K}$.


The cubic norm on $J$ (from which the Jordan algebra structure may be constructed) is then given by (essentially the determinant)
\[
N(x)=abc+\Tr(xyz)-ax\bar{x}-by\bar{y}-cz\bar{z}.
\]
The $\#$ map is given by (essentially the adjugate matrix):
\[
x^{\#}=\begin{pmatrix} bc-|x|^2 & \bar{y}\bar{x}-cz & zx-b\bar{y} \\
                       xy-c\bar{z} & ac-|y|^2 & \bar{z}\bar{y}-ax \\
                       \bar{x}\bar{z} & yz-ax & ab-n(z)
        \end{pmatrix}.
\]
 
\end{example}

Examples \ref{example:cubicnormgeneric} and \ref{example:cubicnorm} cover all the examples of Jordan algebras considered in this paper.

We recall the following definitions in the theory of Freudenthal triple systems, specialized to our case.

\begin{defn}\label{defn:hodgestructure}\hfill
\begin{enumerate} 
\item 
Let $V_J$ be the following $\mb{Q}$-vector space constructed out of $J$: 
\[
V_J\defeq \mb{Q}\oplus J\oplus J^{\vee}\oplus \mb{Q}.
\]
When there is no danger for confusion, we will omit the subscipt and simply denote $V_J$ by $V$.
\item
For  a typical element of $Q\in V_J$, we write  
\[
Q=(p^0, p, q, q_0)
\]
for its components (so that $p^0\in \mb{Q}, p\in J$, and so on). When we have a basis for $J$, and therefore a dual basis for $J^{\vee}$,. we will also write $p=(p^i), q=(q_i)$ for the components of the vectors $p$ and $q$. By rewriting 
\[
V_J =(\mb{Q}\oplus J)\oplus (\mb{Q}\oplus J)^{\vee}
\]
we obtain a symplectic form $\omega$ on $V_J$. 
\item There is a canonical quartic form on $V_J$ (this is the crucial piece of data for a Freudenthal triple system):
\[
I_4\defeq 4p^0\mc{V}(q)-4q_0\mc{V}(p)+4((q_\#, p^{\#})-(p^0q_0+(p,q))^2,
\]
where $(\cdot ,\cdot)$ denotes the canonical pairing between $J$ and $J^{\vee}$, and $p_{\#}\in J$ is the image of $p$ under the quadratic map $\#: J \rightarrow J$. 

\end{enumerate}
\end{defn}
\begin{rmk}\label{rmk:quartic}
A less brutal definition of the quartic form $I_4$ is using root systems: it turns out that  there is a $5$-graded Lie algebra 
\[
\mathfrak{g}=\mathfrak{g}_{-2}\oplus \mathfrak{g}_{-1}\oplus \mathfrak{g}_0\oplus \mathfrak{g}_1 \oplus \mathfrak{g}_2
\]
for which $V=\mathfrak{g}_1$, and both $\mathfrak{g}_{-2}$ and $\mathfrak{g}_2$ are one dimensional. If we then fix basis elements $Y_{-2}$ and $Y_2$ for $\mathfrak{g}_{-2}$ and $\mathfrak{g}_{2}$ respectively, then the quartic form is simply given by the following formula: let $X\in V$, then $I_4(X)$ is defined by 
\[
[X,[X,[X,[X,Y_{-2}]]]=I_4(X)Y_2.
\]
For details of this we refer the reader to  \cite{helenius}.
\end{rmk}

\begin{defn}\label{defn:shimuragroups}
Given a cubic norm structure $N$ on $J$ we define the following two algebraic  groups
\begin{align*}
     M_J&\defeq \{(\lambda, g)\in \mb{G}_m\times \GL(J)|N(gx)=\lambda N(x) \ \text{for all} \ x\in J\};\\
     G_J&\defeq \{(\nu, g)\in \mb{G}_m\times \GL(V_J)|\langle gv, gw\rangle =\nu \langle v, w\rangle \ \text{and} \ I_4(gv)=\nu^2I_4(v) \text{for all } v, w \in V_J\}.
\end{align*}
\end{defn}
\begin{rmk}
We will sometimes also refer to $M_J$ as $G_5$ and $G_J$ as $G_4$, since these are the symmetry groups of the physical theory in $5$ and $4$ dimensions, respectively.
\end{rmk}

\begin{example}\label{example:cubic1}
As an example to keep in  mind,  if we take $J$ to be the split Jordan algebra $\mb{Q}$, then $V$ is 4-dimensional, and $G_J\cong \GL_2(\mb{Q})$, and $V\cong \Sym^3(\std)$ as a representaion of $G_J$. We may also think of $V$ as the space of binary cubic forms over $\mb{Q}$. The 5-graded  Lie algebra alluded to in Remark \ref{rmk:quartic} is the Lie algebra of $G_2$.
\end{example}

\subsection{The hermitian symmetric spaces associated to $G_J$}
For each of the groups $G=G_J$ constructed in the previous section, the symmetric space of the real group is a hermitian tube domain. In fact, the examples exhaust all hermitian tube domains of rank 3: see \cite{gross}.
\begin{defn}
Let $K$ be a maximal compact subgroup of $G(\mb{R})$, and $D\defeq G(\mb{R})/K$ be the symmetric space associated to $G$. 
\end{defn} 
We list the symmetric spaces in each of the cases
\begin{enumerate}
    \item $J=\mb{R}\oplus \Gamma_{1, n-1}, \ D=\SL(2, \mb{R})/\SO(2)\times \O(2,n)/\O(2)\times \O(n)$,
    \item $J=\herm_3(\mb{R}), \ D=\Sp(6, \mb{R})/\U(1)\times \SU(3)$,
    \item $J=\herm_3(\mb{C}), \ D=\SU(3,3)/\U(1)\times \SU(3)\times \SU(3)$
    \item $J=\herm_3(\mb{H}), \ D=\SO^*(12)/\U(1)\times \SU(6)$,
    \item $J=\herm_3(\mb{O}), \ D=E_{7(-25)}/\U(1)\times E_{6(-78)}$.
\end{enumerate}
Note that we follow the notation in \cite{ferraramarrani} and denote by $E_{7(p)}$ the real form of $E_7$ with $p=
    (\#\textrm{non-compact generators})
 - 
    (\#\textrm{compact generators})
$.
Thus the compact real form is $E_{7(-133)}$ in this notation, and similarly $E_{6(-78)}$ is the compact real form of $E_6$. Note also that each of these spaces has hermitian structure, which is of course as  expected since they are open  orbits of the flag variety $G(\mb{C})/P(\mb{C})$. This list  also shows up in Pollack's theory of ''modular forms'' on symmetric spaces which \textit{do not} have hermitian structure: the groups listed above are the levi subgroups of the groups in Pollack's theory. This is not an accident and we will  return to this in a future work; in a sentence the symmetric spaces in Pollack's theory are the moduli spaces of the $3d$ theory obtained from the $4d$ theory by compactifying on a circle. 


The following, which is well known, gives an explicit description of the hermitian symmetric domain $D$, tying it together with the special geometry reviewed in Section \ref{section:specialgeometry}. This point of veiw will be useful when we consider the canonical variation of Hodge structure.


\begin{prop}\hfill
\begin{enumerate}
\item 
The orbit $\mc{O} \subset V_{\mb{C}}$ is a holomorphic Lagrangian cone  and  has a special K\"ahler structure, where the vector bundle $\mc{E}$ is given by the trivial vector bundle with fiber $V_{\mb{C}}$. 
\item If we denote by $\mc{M}_D$ the projectivization of $\mc{O}$ then $\mc{M}_D$ is isomorphic to the flag variety $G_{\mb{C}}/P$, where $P$ is the stabilizer of $(1,0,0,0)$. The action of $G$ on $\mc{M}_D$ has two open orbits, either of which is isomorphic to the Hermitian symmetric tube domain $G/K$; we will denote one of these open orbits by $\mc{M}$.
\item Writing $t^i=\frac{X^i}{X^0}$ for $i=1, \cdots, n$, we have honest coordinates on $\mc{M}$. Writing $t^i=x^i+iy^i$ the real coordinates $x^i, y^i$ give an isomorphism
\[
\mc{M}\cong J+iJ_+
\]
to the K\"ocher generalized upper half plane associated to the Jordan algebra $J$, where 
\[
J_+\defeq \{x\in J(\mb{R})|N(x)>0\}.
\]
\end{enumerate}
\end{prop}
\begin{rmk}
The flag varieties $\mc{M}_D$  are the same as the ones studied by Manivel \cite[Section 1.5]{maniveltopics}. 
\end{rmk}

Note that we obtain  a real variation of Hodge structure on $V_{\mb{R}}$ as follows. There is particular element $D\in \mf{g}$, the (real) lie algebra $G$, and a decomposition (note that this is not the usual Hodge splitting appearing in the theory of Shimura varieties, which is a splitting of the complexified Lie algebra):
\[
\mf{g}=\mf{g}_-\oplus \mf{g}_0 \oplus \mf{g}_+,
\]
where $D$ acts with eigenvalues $-2, 0, 2$ respectively. The eigenspace $\mf{g}_+$ is canonically isomorphic to $J_{\mb{R}}$ itself, and let $T^i$ be the basis of $\mf{g}_+$ corresponding to the basis of $J$ we have fixed. Then using the coordinates $t^i$ from above the Hodge decomposition of $V_{\mb{C}}$ is given by the element 
\[
D_t\defeq \Ad(\exp \sum_i t^i T_i)(D).
\]
In other words, at the point with coordinates $(t^i)$ the Hodge cocharacter of the corresponding $\mb{R}$ Hodge structure 
\[
\mb{S}\rightarrow G
\]
is given by the element $D_t$, where $\mb{S}=\Res_{\mb{C}/\mb{R}}\mb{G}_m$ is the Deligne torus..

\subsection{The associated Shimura varieties}
We may now define the Shimura varieties that are of interest to us. As before let $J$ be a Jordan algebra over $\mb{Q}$, and consider the group $G$ as defined in Definition~\ref{defn:shimuragroups}.
\begin{defn}
For $\mc{U}$ a compact open of $G(\mb{A}^{\infty})$ (here $\mb{A}^{\infty}$ denotes the finite adeles), we define 
\[
\Sh_{\mc{U}}\defeq G(\mb{Q})\backslash D\times G(\mb{A}^{\infty})/\mc{U}.
\]
\end{defn}
We therefore have a projective system of varieties $\Sh_{\mc{U}}$ indexed by compact open subgroups of $G(\mb{A}^{\infty})$. These are known, by the work of many people,  to be algebraic varieties (or more accurately)  Deligne-Mumford stacks with models over number fields called canonical models. We note that if $p$ is an attractor point for the charge vector $\gamma$, and we have an element $g\in G(\mb{Z})$, then $g\cdot p$ is an attractor point for the charge vector $g\cdot \gamma$. Therefore the projection of an attractor point to a Shimura variety depends only on the orbit of the charge vector.

We say a word about the possible $\mb{Q}$-forms of the groups. For the family of type $\SO(2,n)$ this requires a quadratic space over $\mb{Q}$ of signature $(1,n-1)$; for the  cases where the groups are $\SL_2$, $\Sp(6)$ and $E_7$ the descent to $\mb{Q}$ is unique, and the Jordan algebras are simply $\mb{R}, \herm_3(\mb{Q})$ and $\herm_3(\mb{O})$, respectively. Finally, for the two remaining exceptional examples with groups of type $\SU(3,3)$ and $\SO^*(12)$, the descent to $\mb{Q}$ requires a choice of imaginary quadratic field $F$ and a quaternion algebra $B$ (which is non-split over $\mb{R}$), respectively, and the Jordan algebras are again $\herm_3(F)$ and $\herm_3(B)$ respectively.  

\subsection{The canonical variations of Hodge structures}
We may now define the weight 3 Calabi-Yau variations of Hodge structure that we are interested in. Let $V\defeq V_J=\mb{Q}\oplus J\oplus J^{\vee}\oplus \mb{Q}^{\vee}$ be the $\mb{Q}$-vector space underlying the Freudenthal triple system; it is a representation of $G_J$ by construction, and so gives rise to a variation of Hodge structures. Essentially, $V$ gives rise to an equivariant vector bundle with connection over $D$, and the Hodge decomposition of $V$ at a point $x\in D$ is that induced by the Hodge cocharacter
\[
\mb{S}\rightarrow G_J(\mb{R})
\]
associated to the point $x$.
These coincide with those constructed by Gross (for hermitian tube domains of rank 3).
\begin{thm}[{\cite{gross}}]\label{thm:grossvhs}
The vector bundle with connection given by $\mc{V}\defeq V\otimes \mc{O}_D$ underlies a  VHS of CY type, which  is pure  of weight 3.
\end{thm}

\subsection{The attractor conjecture for CYVHS from Jordan algebras}
Since the only thing we need to define the attractor points is the rational variation of  Hodge structure, we may make the following definition  as in the geometric case.
\begin{defn}
A point $x\in D$ is an attractor point if there exists a class $\gamma\in H_x$ such that $x\in H^{3,0}\oplus H^{0,3}$. 
\end{defn}
Thus we have special points in moduli space parametrized by orbits of $G(\mb{Z})$ acting on $V(\mb{Z})$. This should be compared with the construction of special points in $Bun_G$ of a curve over $\mb{F}_q$ by Thorne, again parametrized by orbits of a group acting on a vector space, but in the function field setting. We hope to return to this point in forthcoming work. 


\begin{example}\label{example:cubic2}
Let us continue with our running $\GL_2$-example. In this case, the attractor points are indexed by integer binary cubic forms $\gamma \in V$ with positive discriminant, and the points themselves  are given by Julia's ``covariant map'' \cite{julia} (see also \cite{cremona}), which we describe presently. Since we are assuming positive discriminant, $\gamma$ has three real roots $\alpha, \beta, \gamma$, so we can draw them on the boundary circle of the Poincar\'e disk model. Then there is a unique map from the set of three points on the boundary to a point in the interior which respects the action of $\SL_2(\mb{R})$, which we can describe as follows: since the action of Mobius maps is 3-transitive, map the points $\alpha, \beta, \gamma$ to the third roots of unity. Then we define the image of the map to be $0$, and then we may apply the inverse Mobius map to get the desired point.

It turns out that these points are imaginary quadratic numbers if we consider them in the upper half plane model: that is the coordinate $\tau$ of each of these points lies in some imaginary quadratic field. In fact, these points give precisely the 3 torsion elements in the class group if we identify the imaginary quadratic numbers with elements of the $\Cl (\mb{Q}(\sqrt{D}))$, where $D$ is the discriminant.
\end{example}

We may now state the following theorem, which implies  the attractor conjecture for the VHS of CY type that we study, the proof of which will occupy the rest of this section.
\begin{thm}\label{thm:cmattractor}
Let $\gamma \in V$ be a charge vector such that $I_4(\gamma)>0$ (this is called the BPS condition). Then any  attractor point associated to $\gamma$ is a CM point.
\end{thm}

\subsection{Hodge theory and special geometry}
In this section we describe how a VHS of  CY type of weight 3 gives rise to a special geometry. We then specialize to the VHS defined in the previous section and describe the geometry explicitly.

Suppose $\ms{V}\rightarrow S$ is a variation of Hodge structure of CY type of weight 3 over a simply connected base complex manifold $S$, and that $\ms{V}$ has dimension $2n+2$. By definition $\ms{V}$ has a flat symplectic structure $\langle, \rangle$, and so locally we may pick a symplectic basis $\gamma^0, \gamma^I, \gamma_I, \gamma_0$ where $I=0, \cdots, n$: that is, the basis $\gamma^I, \gamma_I$ satisfy
\[
\langle \gamma^I, \gamma^J \rangle = \langle \gamma_I, \gamma_J \rangle =0, \
\langle \gamma^I, \gamma_J \rangle =\delta^I_J.
\]

Fix a base point $p_0\in S$, and let $V\defeq \ms{V}_{p_0}\otimes\mb{C}$ denote the complexified fiber of $\ms{V}$ at $p_0$. Then we have an immersion 
\[
\iota: S\rightarrow \mb{P}(V)
\]
sending a point $p$ to the line $H^{3,0}_p\subset V$ (that is, the $(3,0)$-piece of the Hodge decomposition at $p$). This is the map giving us the special geometry, as in Section \ref{section:specialgeometry}. We record the following result (this is stated, for instance in \cite{heateqn} (see Equation (3.6) and the paragraph preceding it), and proved in \cite{friedmanlaza}). It is more convenient to introduce an auxiliary space (known as the conical special K\"ahler space). Define  
\[
\tilde{S}\defeq \{(p, v)|p\in S, \ v\in H^{3,0}|_p\backslash \{0\}\}; 
\]
that is, $\tilde{S}$ is $S$ equipped with a trivialization of $H^{3,0}$. Then the map $\iota$ lifts to 
\[
\tilde{\iota}: \tilde{S}\rightarrow V.
\]

\begin{prop}
The map $\tilde{\iota}$ has the following local description: let $X^I, F_I$ be the coordinate functions of $\tilde{\iota}$ in the basis $\gamma^I, \gamma_I$ chosen previously. Then locally there exists a holomorphic function $F(X^I)$ such that 
\[
F_I=\frac{\partial F}{\partial X^I}.
\]
In other words, the VHS $\ms{V}$ gives rise to a special geometry; in particular $\tilde{\iota}(\tilde{S})$ is a conical Lagrangian inside $V(\mb{C})$, and $\iota(S)\subset \mb{P}(V)(\mb{C})$ is Legendrian.

\end{prop}

We now specialize to the weight 3 VHS over hermitian symmetric tube domains introduced in Section \ref{} attached to a Jordan algebra $J$. We choose   a   symplectic basis $\gamma^I, \gamma_I$ for the symplectic space 
\[
V=\mb{Q}\oplus J\oplus J^{\vee}\oplus \mb{Q}^{\vee}
\]
such that $\gamma^0$ is a basis for $\mb{Q}$, $\gamma^I$ a basis for $J$, and so on. Then we have 
\begin{lem}
For each of the weight 3 VHS over hermitian tube domains, the function $F(X^I)$ is given by  
\[
F(X^I)=\sum_{i,j,k}N_{ijk}\frac{X^iX^jX^k}{X^0},
\]
where the $N_{ijk}$ are the structural constants for the cubic norm structure for $J$.
\end{lem}

\begin{proof}
 By \cite[Theorem 6.5]{friedmanlaza} and a routine calculation we know that we may take 
 \[
 F\defeq \frac{\varphi(X^i)}{X^0},
 \]
 where $\varphi$ is a cubic polynomial in the $X^i$'s (following the notation of loc.cit.). Note that by the same theorem we have that $X^0$ is never zero, since our $\gamma^0$ is the same as $f_0$ of loc.cit. It remains to identify $\varphi$ with the cubic norm of $J$ This boils down to the question of what is the explicit form of the  function $\varphi$ (the generating function in symplectic geometry language) defining the Legendrian subvariety $\iota(S)\subset \mb{P}(V)$, or equivalently the orbit of the vector 
 \[
 (1,0,0,0)\in V=\mb{Q}\oplus J\oplus J^{\vee}\oplus \mb{Q}^{\vee}.
 \]
 The fact that this function is given by the cubic norm of the Jordan algebra is a classical fact: for example, we refer the reader to \cite[Equation (7.4)]{gunaydinlectures}.
\end{proof}


\subsection{Existence and uniqueness of attractor points}
In this section we show for each BPS charge vector $\gamma \in V$, an attractor  point exists and is unique.

\begin{prop}Let $\gamma \in V$ be a vector satisfying $I_4(\gamma)>0$. Then the following hold:
\begin{enumerate} 
\item an attractor point for $\gamma$ exists, and furthermore,
\item  such an  attractor point is unique.
\end{enumerate}
\end{prop} 
\begin{proof}
Certainly an attractor point exists for at least one choice of  $\gamma_0 \in V$, which we denote by $p_0$. Now suppose $\gamma$ is another choice of BPS charge vector. The  group $G(\mb{R})$ acts transitively on the cone 
\begin{equation}\label{eqn:positivecone}
\mc{C}\defeq \{v\in V(\mb{R})|I_4(v)>0\},
\end{equation}
and so we can find $g\in G(\mb{R})$ satisfying $g\cdot \gamma_0=\lambda \gamma$, for some $\lambda\in \mb{R}_{>0}$. Then the following claim will prove the proposition:
\begin{claim}
The point $p\defeq g\cdot p_0$ is an attractor point for the charge vector $\gamma$.
\end{claim}
\begin{proof}[Proof of Claim]
Let $\widetilde{p}_0\in \tilde{\mc{M}}$  be a lift of $p_0$, and let $\omega$ denote the image  of $\tilde{\iota}(\tilde{p}_0 \in V(\mb{C})$. By scaling the choice of $\Omega \in H^{3,0}|_{p_0}$ if necessary, we may assume that the  attractor condition reads
\[
\Re(\Omega)=\gamma_0.
\]
Then acting by $g$ we have 
\[
\Re(g\cdot \Omega)=\lambda \gamma,
\]
and since $g\cdot \Omega/\lambda  =\tilde{\iota}(\tilde{p})$ for an appropriate lift $\tilde{p} \in \tilde{\mc{M}}$ of $p$, we are done. 

\end{proof}

Now we prove (ii), i.e. that the attractor point is unique. For this, note that the formulas in Lemma \ref{lemma:hamiltonian} actually hold globally. To see this, note that $\tilde{M}$ is embedded in $V(\mb{C})$ \cite{friedmanlaza}, and let us consider the map 
\begin{align*}
\pi: \widetilde{M}&\rightarrow V(\mb{R})\\
\Omega &\mapsto \Re (\Omega);
\end{align*}
in coordinates this sends a point $\Omega \in \tilde{M}$ to $(p^I, q_I)$ in the notation of Section \ref{section:realdarboux}. Then this map lands in the cone $\mc{C}$ (defined above in (\ref{eqn:positivecone})), and it has a section $\sigma$ given by the formulas in Lemma \ref{lemma:hamiltonian}. This shows that the map is an isomorphism, and hence a point in $\tilde{M}$ is determined by the real parts  $p^I=\Re (X^I), q_I=\Re (F_I)$, which are determined by the attractor condition, so we have shown uniqueness.
\end{proof}

We will use the following criterion for CM Hodge structures:
\begin{prop}\label{prop:rigid}
For a Calabi-Yau variation of Hodge structure $W$ corresponding to a point $p\in \mc{M}$,  if the subspace
\[
H^{3,0}\oplus H^{0,3}\subset W\otimes \mb{C}
\]
is in fact a rational sub-Hodge structure, then $p$ is a CM point: that is, the Mumford-Tate group of $W$ is a torus.
\end{prop}
\begin{proof}
Let $M$ denote the Mumford-Tate group of $W$. It suffices to prove that the group   $M(\mb{R})$  stabilizes the point $p$: indeed this implies the Hodge structure is CM  by  \cite[page 136 V.4]{ggk}. Now by definition of the Mumford-Tate group, any point in the orbit will also have $H^{3,0}\oplus H^{0,3}$ as a sub rational Hodge structure. But by the defining property of the canonical VHS,  infinitesimally the section $\Omega $ will deform to a class in $h^{2,1}$, and therefore the orbit must be zero dimensional, i.e. the point $p$ itself, as required.
\end{proof}

\subsection{Proof of the attractor conjecture for the Jordan theories}
We may now give a proof of the attractor conjecture in the case of the CYVHS attached to Jordan algebras.

For this we use a formula for the attractor points found by Ferrara  \cite{composite}. Pick a symplectic basis $\gamma_I, \gamma^I$ where $I=0, 1, \cdots, n$of the vector space $\mb{Q}\oplus J\oplus J^{\vee}\oplus \mb{Q}$ as before, where $I=0, 1, \cdots, n$,  and suppose we have a charge vector 
\[
\gamma=\sum_I p^I\gamma_I+\sum_Iq_I \gamma^I,
\]
where the $p^I, q_I$'s are integers. Writing the class $\Omega \in H^{3,0}$ as 
\[
\Omega=X^I\gamma_I+F_I\gamma^I
\]
gives us  homogeneous coordinates $X^I$ for the moduli space $\mc{M}$, and let 
\[
t^I=X^I/X^0
\]
where $I=1, \cdots , n$,  be honest coordinates (on the universal cover, say). Furthermore let $I_4(p,q)$ be the quartic invariant defined in Section~\ref{subsection:jordan}, and $I_1(p,q)=\sqrt{I_4(p,q)}$ (by abuse of notation we will sometimes write a function $f$ of the integers $p^I, q_I$ simply as  $f(p,q)$). 
We begin with the following
\begin{prop}\label{prop:hessequartic}
The coordinates of the attractor point of charge $\gamma$ are given by 
\[
t^I=\frac{p^I+i\frac{\partial I_1}{\partial q_I}}{p^0+i\frac{\partial I_1}{\partial q_0}}.
\]
In other words, the Hamiltonian  $S(p,q)$ in Lemma \ref{lemma:hamiltonian} is given simply by 
\[
S=\sqrt{I_4}.
\]

\end{prop}
\begin{proof}
We refer the reader to \cite[Equation (3.45)]{composite} for this computation.
\end{proof}
\begin{rmk}
We should make a remark on the notation here: even though the $p^I, q_I$'s will only be taken to be integers in everything that follows, the notation $\frac{\partial I_1}{\partial q_I}$ means that we treat the function $I_1$ formally as a function of the  $p^I$'s and $q_I$'s, take the partial derivative formally, and then evaluate back at the integers $p^I, q_I$.
\end{rmk}
\begin{proof}[Proof of Theorem \ref{thm:main1}]
Recall that, by Proposition~\ref{prop:rigid}, it suffices to show that at an attractor point, the subspace $H^{3,0}\oplus H^{0,3}$ is a rational sub-Hodge structure.  Write $\gamma=\Omega+\bar{\Omega}$, where $\Omega \in H^{3,0}$ is some non-zero class.

Now by Proposition~\ref{prop:hessequartic}, we see that, by scaling the element $\Omega \in H^{3,0}$ if necessary, we may assume that 
\[
X^0=p^0+i\frac{\partial I_1}{\partial q_0}, \ X^I=p^I+i\frac{\partial I_1}{\partial q_I}
\]
Now recall  that the quartic form $I_4$ has integer coefficients and note that we have  
\[
\frac{\partial \sqrt{I_4(p,q)}}{\partial q_I}=\frac{1}{2\sqrt{I_4(p,q)}}\frac{\partial I_4}{\partial q_I}.
\]
Therefore  the  homogeneous coordinates $X^I$ all satisfy 
\[
X^I\in \mb{Q}(\sqrt{D})
\]
where $D$ is a negative integer, since for BPS attractors  $I_4>0$, and hence we have 
\[
D\defeq \big(i\sqrt{I_4(p,q)}\big)^2 \in \mb{Z}_{<0}.
\]
The same is true for the $F_I$'s, since by definition
\[
F_I=\frac{\partial F}{\partial X^I},
\]
and recall that the prepotential is given by
\[
F=\sum_{i,j,k}d_{ijk}\frac{X^iX^jX^k}{X^0};
\]
in other words we have 
\[
F_I\in \mb{Q}(\sqrt{D})
\]
as well.
Now consider the element 
\[
\Omega'\defeq \sqrt{D}\Omega \in H^{3,0}
\]
instead of $\Omega$. We want to show that $\Omega'+\overline{\Omega}'$ is a rational class as well, at which point we will be done. In other words, we need to show 
\[
\Re (\Omega') \in \mb{Q}.
\]
On the other hand
\begin{align*}
\Re (\Omega')&=\Re (X^I\gamma_I+F_I\gamma^I)\\
           &=\Re(\sqrt{D}X^I)\gamma_I+\Re(\sqrt{D}F_I)\gamma^I\\
\end{align*}
and 
\[
\Re(\sqrt{D}X^I), \ \Re(\sqrt{D}F_I)\in \mb{Q}
\]
as required.
\end{proof}

\begin{rmk}
We note that the same integral orbits have been studied by Bhargava and, more relevant for us, by Pollack in the twisted case. Starting with a Jordan algebra as we did, Pollack studies the prehomogenous space $W_J$ and shows that the orbits parametrize quadratic rings $S$ along with a fractional ideal inside $S\otimes J$. The fraction field of $S$ is exactly the CM field obtained in \ref{thm:cmattractor}. It would be interesting to investigate further the relation between the refined information, namely the integrality as well as the fractional ideals, obtained by Bhagarva and Pollack and the Hodge structures obtained here.
\end{rmk}

\subsection{Restriction of black hole potential to hermitian sub-domains}
For each of the hermitian symmetric domains $\mc{M}$ considered in the previous section and a hermitian sub-domain $\mc{M}'$, we may restrict the black hole potential function associated to a charge vector $\gamma$ and try to find the critical points of this restricted function. For example, for a $E_6$ domain (whose CYVHS is of Type II, i.e. does not admit a real structure) embedding inside the $E_7$ domain, we recover the attractor points on the $E_6$ domain as before. 

In fact, this embedding actually arises as the \textit{moduli space} of another class of attractors, known as the non-BPS $Z=0$ attractors.

\subsection{Very special SUGRAs from $T$-algebras}\label{subsection:talg}

There is a close and more exotic analogue of the Jordan theories considered in this section, where one replaces Jordan algebras by so-called $T$-algebras (see for example \cite{homogeneousspecial}). The moduli spaces are still bounded domains  and still possess a transitive action of a group (as in Section \ref{subsection:jordan}); the  difference is that the group acting is no longer semisimple. The techniques above should apply in this situation as well: we will investigate these theories in future work.



\section{Non-BPS attractors in the $t^3$-model}\label{section:nonbps}
It turns out that the kind of attractors, especially in the context of Section \ref{section:CYVHS}, we have been considering so far all lie in the so-called (in the physics literature) BPS (Bogomol'nyi-Prasad-Somerfield) sector. Mathematically  this  corresponds to a positive discriminant condition; for example in the cubic example, also known as the $t^3$-model (see Examples \ref{example:cubic1} and \ref{example:cubic2}) we only looked at those binary cubic forms with positive discriminant, i.e. those with three positive real roots. In this section we explicitly compute  the \textit{non-BPS} attractors in this one modulus $t^3$-model. 

\subsection{Description of the non-BPS attractors}

The main result of this section is the following

\begin{thm}
\label{thm:nonbpst^3}
The non-BPS attractors in the $t^3$-model are indexed by binary cubic forms with negative discriminant. The attractors are given by the covariant map of Julia \cite{julia} (also see \cite{cremona, cremonastoll}).
\end{thm}

Before proving the Proposition we must first  recall the definition of non-BPS attractors in the general case. For this we need to recall some notation from ''special geometry'', which is the name of certain special coordinates which exist on all of the moduli spaces that we have been considering in this paper. 

As in the BPS case, we fix an integral class $\gamma$: so in the geometric context (i.e. if we have an honest family of Calabi-Yau varieties) $\gamma$ is a class in Betti cohomology, and in the VHS case we have $\gamma \in V$ where $V$ is the vector space underlying the rational VHS, defined as in \ref{defn:hodgestructure}. Then given  a global holomorphic volume form $\Omega$ (or a trivialization of $H^{3,0}$ in the VHS case), we can define the central charge 
\[
Z\defeq \langle \gamma, \Omega \rangle.
\]
Note that this is a function on the universal cover of the moduli space, or we may view it as a section of a line bundle $\mc{L}\rightarrow \mc{M}$. The line bundle $\mc{L}$ comes equipped with a covariant derivative, which is given by the formula 
\[
D_iZ=\partial_i Z+\partial_iK \cdot Z,
\]
where $K$ is the Kahler potential defined by 
\[
e^{-K}=i \langle \Omega, \overline{\Omega} \rangle.
\]
Finally we also  have the \textit{rescaled central charge} 
\[
\mc{Z}\defeq e^{K/2}Z,
\]
which does not depend on the choice of $\Omega$, and the metric 
\[
g_{i\Bar{\jmath}}=\partial_i \partial_{\Bar{\jmath}} \log |\mc{Z}|^2.
\]

We can now say what we mean by an attractor point (BPS or not):
\begin{defn}[\cite{fgk}]\label{defn:nonbpsattractor}
Consider the potential 
\[
V_{eff}\defeq e^K[g^{i\Bar{\jmath}}(D_iZ)(\overline{D}_{\Bar{\imath}}\overline{Z})+|Z|^2].
\]
The attractors are the critical points of $V_{eff}$. If furthermore at a critical point
\begin{enumerate} 
\item  the Hessian of $V_{eff}$ is positive definite, then we call it a BPS attractor; 
\item otherwise we call it   non-BPS. 
\end{enumerate}
\end{defn}

\begin{rmk}
Note that the definition of BPS attractors is the same as \ref{defn:attractor}, since these minimize $|\mc{Z}|^2$ also. In fact, they minimize the two terms of $V_{eff}$ separately.
\end{rmk}

If we have an integral structure on the variation of Hodge structure $V$, then a more conceptual  definition for the effective potential is as follows (again we fix a class $\gamma \in V$): at each point $x$ in moduli space the period matrix of the Griffiths intermediate Jacobian gives a quadratic form on $V$, and the value of $V_{eff}(x)$ is precisely this  quadratic form evaluated at the class $\gamma$. Indeed $D_i\Omega$ gives a basis of $h^{2,1}$ and $D_iZ$ is the integral of $\gamma$ against this basis. Then the first term in the expression for $V_{eff}$ is the contribution of the quadratic form coming from $H^{2,1}$ and $H^{1,2}$ and the second term is that coming from $H^{3,0}$ and $H^{0,3}$. 

From this discussion it is also clear that the BPS attractors in \ref{defn:nonbpsattractor} is the same as the usual cohomological one: again $D_iZ $ is the coefficient of $\gamma$ attached to $D_i\Omega$ when we write it in the basis $\Omega, \overline{\Omega}, D_i\Omega, \overline{D_i\Omega}$ of $V\otimes \mb{C}$.

The above definition is for general families of Calabi-Yau threefolds (and variations of Hodge structures); in the special cases we study, the definition simplifies.
\begin{defn}\label{defn:bpscharge}
For a charge vector $\gamma$ with $I_4(\gamma)$, we say that it is 
\begin{enumerate} 
\item \emph{BPS} if $I_4(\gamma)>0$, 
\item \label{item:nonbps} \emph{non-BPS} if $I_4(\gamma)<0$, 
\end{enumerate}
\end{defn}
\begin{rmk}
There is also the possibility that $I_4(\gamma)=0$, in which case it is referred to as \emph{non-BPS with $Z=0$} in the physics literature; since we will not deal with such $\gamma$'s in this work, we make no mention of this.
\end{rmk}
\begin{prop}[{\cite[p.43]{chargeorbit}}]
The two notions of BPS versus non-BPS attractors in Definition~\ref{defn:nonbpsattractor} and Definition~\ref{defn:bpscharge} agree. Furthermore, at a non-BPS attractor point (with $Z\neq 0$, as in the rest of this paper), there are $n+1$ positive eigenvalues of the Hessian, and $n-1$ zeros, where $n=\dim_{\mb{C}} \mc{M}$. 
\end{prop}

As a warm up for the proof of Theorem \ref{thm:nonbpst^3}, we note that the analogue of Theorem \ref{thm:nonbpst^3} also holds for the BPS attractors. That is, we have 
\begin{prop}
The BPS attractors in the $t^3$-model are indexed by binary cubic forms with positive discriminant, and the attractor point is given by Julia's covariant map.
\end{prop}
\begin{proof}
For a BPS charge vector $\gamma=(p^0,p^1, q_0, q_1)$, denote the associated cubic by 
\[
F(\tau)=p^0+p^1\tau +q_1\tau^2-\frac{q_0}{3}\tau^3.
\]
Now using the formula
\begin{align}
\langle \Omega , \overline{\Omega}\rangle &=X^0\bar{F}_0-\bar{X}^0F_0+X^1\bar{F}_1-\bar{X}^1F_1.
\end{align}
as well as the formula $F_I=\frac{\partial F}{\partial X^I}$ with $F(X)=(X^1)^3/3X^0$, we deduce that the potential is 
\begin{align}
|\mc{Z}_{\gamma}|(\tau)&=\frac{|p^0X^0+p^1X^1+q_0F_0+q_1F_1|}{|X^0\bar{F}_0-\bar{X}^0F_0+X^1\bar{F}_1-\bar{X}^1F_1|^{\frac{1}{2}}}\\
&=\frac{|p^0X^0+p^1X^1+q_0F_0+q_1F_1|}{\Big|-X^0\frac{(\bar{X}^1)^3}{3(\bar{X}^0)^2}+\bar{X}^0\frac{(X^1)^3}{3(X^0)^2}+X^1\frac{(\bar{X}^1)^2}{\bar{X}^0}-\bar{X}^1\frac{(X^1)^2}{X^0}\Big|^{\frac{1}{2}}}\\
&=\frac{|F(\tau)|}{(\Imag \tau)^{\frac{3}{2}}}.
\end{align}
Note that, as usual, we let $\tau=X^1/X^0$. Now this is precisely the function whose minimum defines Julia's  covariant map: see \cite[Proposition 5.1]{cremonastoll}, so we conclude that the attractor point is given by Julia's covariant map. 
\end{proof}

We can now prove Proposition \ref{thm:nonbpst^3}.
\begin{proof}[Proof of Theorem \ref{thm:nonbpst^3}]
Recall that the Julia covariant map is obtained by minimizing the following function on $\mb{H}_3$:
\begin{equation}\label{eqn:f1}
\log(F_1)\defeq \log \frac{|t-\alpha|^2+u^2}{u}+\log \frac{|t-\beta|^2+u^2}{u}+\log \frac{|t-\gamma|^2+u^2}{u};
\end{equation}
here $(t,u)\in \mb{H}_3$ where $t\in \mb{P}^1(\mb{C})$ is the coordinate of the projection onto the ''floor'' of $\mb{H}^3$ and $u \in \mb{R}$ denotes the ''height'' of the point.

There is a unique point in $\mb{H}_3$ minimizing $F_1$ \cite{}. Since we are assuming the coefficients of the binary cubic form are integers (so in particular real), by symmetry this unique minimum must have $t\in \mb{R}$. So in particular  if we consider the embedding $\mb{H}_2\subset \mb{H}_3$ given by 
\[
\tau=x+iy\mapsto (x, y),
\]

we are looking for the minimum of $F_1$ restricted to $\mb{H}_2$. 

On the other hand the attractor point can be obtained by minimizing a ''fake superpotential''. We first consider a simplified $D0-D6$ system which corresponds to binary cubics of the form
\[
f(X)=pX^3+q.
\]
Then according to \cite[Equation (5.20)]{stu}  the non-BPS attractor with $Z\neq 0$ is obtained \footnote{The superpotential written down in \cite[Equation (5.20)]{stu} is actually that of the more general $stu$-model; to obtain the $t^3$-model that we are interested in we just have to set 
\[
z_1=z_2=z_3,
\]
and 
\[
\alpha_1=\alpha_2=\alpha_3=0,
\]
since the $\alpha_i$'s have to be equal by symmetry and are also required to satisfy $\alpha_1+\alpha_2+\alpha_3=0.$}  by minimizing
\[
F_2\defeq \frac{1}{4}\frac{1}{y^{3/2}}|q^{1/3}+p^{1/3}\tau|^3\Bigg(1+3\frac{(q^{2/3}-p^{2/3}|\tau|^2)^2-p^{2/3}q^{2/3}(\tau-\Bar{\tau})^2}{|q^{1/3}+p^{1/3}\tau|^4}\Bigg),
\]
where $\tau=x+iy$ is the coordinate on $\mb{H}$.

\begin{rmk}
Note that in the non-BPS setting, the attractor point is no longer obtained by minimizing $|\mc{Z}|^2$, as in the BPS case. In particular, as in the  example in \cite[Section 2.7]{moore}, in the non-BPS case $|\mc{Z}|^2$ does achieve a minimum at a point in the upper half plane, namely the unique complex root of the  $Z$ (which by the non-BPS assumption is a cubic with precisely one real root and a pair of complex conjugate roots) in the upper half plane; however this is \textit{not} the attractor point (nor is this claimed to be the case in loc. cit.). 
\end{rmk}

So for $D0-D6$ systems it suffices to prove the following
\begin{claim}
The potentials $F_1$ and $F_2^2$ agree up to a constant.
\end{claim}

\begin{proof}[Proof of Claim]
We begin by simplifying $F_2$: note that we may simplify as follows:

\begin{align*}
(q^{2/3}-p^{2/3}|\tau|^2)^2-p^{2/3}q^{2/3}(\tau-\Bar{\tau})^2 &= 
q^{4/3}+p^{4/3}|\tau|^4-p^{2/3}q^{2/3}(\tau^2+\Bar{\tau}^2)\\ 
&= (p^{2/3}\Bar{\tau}^2-q^{2/3})\cdot (p^{2/3}\tau^2-q^{2/3})\\
&=(p^{1/3}\Bar{\tau}+q^{1/3})\cdot (p^{1/3}\Bar{\tau}-q^{1/3})\cdot (p^{1/3}\tau-q^{1/3})\cdot (p^{1/3}\tau+q^{1/3})\\
&=|q^{1/3}+p^{1/3}\tau|^2\cdot |q^{1/3}-p^{1/3}\tau|^2,\\ 
\end{align*}
and therefore 
\begin{align}\label{eqn:f2product}
F_2&=\frac{1}{4}\frac{1}{y^{3/2}}|q^{1/3}+p^{1/3}\tau|^3
\Bigg(1+3\frac{|q^{1/3}-p^{1/3}\tau|^2}{|q^{1/3}+p^{1/3}\tau|^2}\Bigg)\\
&= \frac{1}{4}\frac{1}{y^{3/2}}|q^{1/3}+p^{1/3}\tau| \Big[(q^{1/3}+p^{1/3}\tau)(q^{1/3}+p^{1/3}\Bar{\tau})+3(q^{1/3}-p^{1/3}\tau)(q^{1/3}-p^{1/3}\Bar{\tau})\Big]\\
&= \frac{|q^{1/3}+p^{1/3}\tau|\cdot \Big[4p^{2/3}\tau \Bar{\tau} -2p^{1/3}q^{1/3}(\tau+\Bar{\tau})+4q^{2/3}\Big]}{4y^{3/2}}.
\end{align}

On the other hand Julia's potential is
\begin{equation}\label{eqn:f1product}
F_1=\frac{(|x-\alpha|^2+y^2)(|x-\beta|^2+y^2)(|x-\gamma|^2+y^2)}{y^3},
\end{equation}
where, as discussed above, we are considering the restriction of the potential $F_1$    in Equation \ref{eqn:f1} on $\mb{H}_2\subset \mb{H}_3$ via the map
\[
\tau=x+iy\mapsto (t,u)=(x,y)\in \mb{H}_3.
\]

Now we use our assumption that the cubic we are considering is $f=pX^3+q$, which has one real root and a pair of complex conjugate roots. We may assume that $\alpha$ is real and $\beta=\Bar{\gamma}$. Then the first term in the numerator in Equation  \ref{eqn:f1product} is simply 
\[
|\tau-\alpha|^2=|\tau+q^{1/3}/p^{1/3}||\Bar{\tau}+q^{1/3}/p^{1/3}|.
\]

On the other hand, the second and third terms are equal since $\beta=\Bar{\gamma}$ and $x$ is real; furthermore 
\[
|x-\beta|^2=(x-\beta)(x-\Bar{\beta})=x^2-(q/p)^{1/3}x+{q/p}^{2/3}
\]
using the fact that $pX^3+q=p(X-\alpha)(X-\beta)(X-\Bar{\beta}).$ 
Hence the product of the last two factors in the numerator of Equation \ref{eqn:f1product} combine to give
\[
(x^2-(q/p)^{1/3}x+{q/p}^{2/3}+y^2)^2=\big[|\tau|^2-(q/p)^{1/3} \big(\frac{\tau+\Bar{\tau}}{2}\big)+(q/p)^{2/3}\big]^2,
\]
which is precisely the square of the  second term of $F_2$ in Equation \ref{eqn:f2product} up to the constant $(4p^{2/3})^2$. Hence we see that $F_1$ and $F_2^2$ agree up to a constant, as required.
\end{proof}

This proves the desired result for $D0-D6$ systems. Now we can use the action of the symmetry group to show  that the same is true for all non-BPS attractors with non-vanishing central charge. This is an argument borrowed from \cite{stu} which we now review for the reader's convenience (see p.17 of loc.cit.).

First note that the effective potential $V_{eff}$ makes sense even when  the $p_i$'s and $q_i$'s are real (as opposed to being integral), as does Julia's covariant map. We want to show the more general statement that these two definitions agree. Now consider the two maps from binary cubic forms to points on $\mb{H}$: the first sends a form to the attractor point (i.e. the minimum of $V_{eff}$), and the second is simply Julia's covariant map. Note that both respect the action of $\SL_2(\mb{R})$. Finally the action of $\SL_2(\mb{R})$ acts transitively on binary cubic forms with the same discriminant, and so we would have proven Lemma \ref{lem:nonbpst^3} if for each $D0-D2-D4-D6$ charge we can find an element in $\SL_2(\mb{R})$ sending it to a $D0-D6$ charge. We now show that this is indeed possible. Suppose we want to map the $D0-D2-D4-D6$ charge $(p^0, p^1, q_1, q_0)$ to the $D0-D6$ charge $(p,q)$. Then we can take the matrix 
\[
M=\frac{-\sgn(\xi)}{\sqrt{(\zeta+\rho)}}\begin{bmatrix}\zeta \xi & -\rho \\
\xi & 1
\end{bmatrix},
\]
where
\[
\zeta=\frac{\sqrt{-\mc{I}}+(p^0q_0-p^1q_1)}{2p^0p^1-p^0q_1},
\]
\[
\rho=\frac{\sqrt{-\mc{I}}-(p^0q_0-p^1q_1)}{2p^0p^1-p^0q_1},
\]
and 
\[
\xi=\Big(\frac{p}{q}\Big)^{1/3}\Bigg[ \frac{2(p^1)^3+p^0(\sqrt{-\mc{I}}-(p^1q_1+p^0q_0)}{2(p^1)^3-p^0(\sqrt{-\mc{I}}-(p^1q_1+p^0q_0)}\Bigg].
\]
It is a straightforward but tedious check that this transformation does exactly what we want. Note that the expression for $\xi$ involves $p$ and $q$: the logic is that we write the discriminant as 
\[
\mc{I}=-(pq)^2
\]
for \textit{some}choice of $p,q$ (which are not necessarily integers, and then perform the transformation $M$ above to send $(p^0, p^1, q_0, q_1)$ to $(p,q)$.

\end{proof}

\begin{rmk}
From the argument above we see that we may equivalently describe these  attractor points in the following way: first fix the attractor point $x$ for some choice of  charge vector $(p^0, p^1, q_0, q_1)$ (using the Julia covariant map, say). Then for another choice of charge vector  $(P^0, P^1, Q_0, Q_1)$, take the unique element $M\in \SL_2(\mb{R})$ which sends first charge vector to the second, and then the attractor point for $(P^0, P^1, Q_0, Q_1)$ is precisely the image of $x$ under $M$. Thus the content of Lemma \ref{lem:nonbpst^3} is that the statement is true on the nose, not just up to an element of $\SL_2(\mb{R})$.
\end{rmk}
We can now give an explicit formula for these non-BPS attractor points:
\begin{cor}
Suppose we have a non-BPS charge vector  $Q=(p^0, p^1, q_0, q_1)$, corresponding to the cubic 
\[
p^0+p^1x+q_1x^2-\frac{q_0}{3}x^3;
\]
also let $\alpha$ denote the unique real root of this cubic. Then the attractor point is given by the root in the upper half plane of the quadratic polynomial
\[
h_0X^2+h_1X+h_2,
\]
where 
\begin{align*}
    h_0&=q_0^2\alpha^2-2q_0q_1\alpha -2p^1q_0-q_1^2\\
    h_1&=-2q_0q_1\alpha^2+(6q_1^2+2q_0p^1)\alpha +2p^1q_1\\
    h_2&=-p^1q_0\alpha^2+3(p^1q_1+p^0q_0)\alpha +2(p^1)^2-3p^0q_1.
\end{align*}
\end{cor}
\begin{proof}
Immediate from the description of the attractor point from Theorem \ref{thm:nonbpst^3} and  \cite[Equation (11)]{cremona} which gives a formula for the covariant map.
\end{proof}

\subsection{Distribution of non-BPS attractors}\label{subsection:t^3distribution}

It is known that the 3-torsion elements of the class groups equidistribute on the upper half plane \cite{hough}, and therefore it should follow that the BPS attractors  equidistribute in the $t^3$-model. On the other hand, there are heuristics (due to Denef-Douglas \cite{denefdouglas} but see also the review of the method in \cite{kachruetc}) to suggest that the distribution of these non-BPS attractors, now coming from integer binary cubic forms with negative discriminant, limit to some exotic distribution on the upper half plane. We can write down explicitly the conjectural distribution in the case of the $t^3$-model, due to \cite{kachruetc}.

Let 
\[
\rho \defeq \frac{1}{8\pi^2}\int |dY^2|e^{-V_{eff}} \frac{\det d^2V_{eff}}{|Y|^2}\theta (\det d^2 V_{eff}),
\]
where 
\[
\det d^2 V_{eff} {|Y|^2}= (4+10|\mc{F}|^2+\frac{9}{4}|\mc{F}|^4-|D\mc{F}|^2)|Y|^2+[2(D\mc{F})\overline{\mc{F}}^2Y^2 +{c.c.}],
\]
and $c.c.$ means taking the complex conjugate,  $\mc{F}(t)=\frac{t^3}{3}$ is the prepotential of the $t^3$ model. Also $\theta$ denotes the Heaviside theta function. Note that $\rho$ depends on where you are in moduli space (at least a priori from the expression we have for it). Then the conjecture is 
\begin{conj}
The non-BPS attractors equidistribute according to the density $\rho$. More precisely, for $f$ a function on the standard fundamental domain for the action of $\SL_2(\mb{Z})$ on $\mb{H}_2$ we have 
\[
\lim_{D \rightarrow \infty}\frac{\sum_{x\in \Gamma_D}f(x)}{|\Gamma_D|} = \int_F f\rho,
\]
where $\Gamma_D$ denotes the set of non-BPS attractors in $F$ of discriminant $D$ (i.e. the points given by Proposition \ref{prop:nonbpst^3}).
\end{conj}

\section{Non-BPS attractors for very special theories}\label{section:nonbpsgeneral}
The picture in Section \ref{section:nonbps} should generalize to the very special theories considered in Section \ref{subsection:jordan} (and probably Section \ref{subsection:talg} as well). However now we no longer expect to get points in moduli space, but rather positive dimensional loci, which are themselves symmetric spaces (but which do not have  hermitian structures). Certain (somewhat imprecise) versions of the results  in this section have been stated in the physics literature, though we are not aware of existing proofs; however, it should be clear the results here are directly inspired by the physics literature.

\subsection{Main result for non-BPS attractors}
\begin{lem}\label{lemma:equivariance}
We have the following equivariance property of $V_{\gamma}$: for $z\in \mc{M}$,
\[
V_{g\cdot \gamma}(g\cdot z)=V_{\gamma}(z).
\]

\end{lem}
\begin{proof}
See for example \cite[Equation (26)]{firstorderflow}.
\end{proof}
We will often refer to this lemma as the ``equivariance'' property of the attractor mechanism with respect to the $G$-action.
\begin{thm}
For $\gamma=(p^I, q_I)\in V$ a non-BPS charge vector, the effective potential $V_{eff}$ (see the formula in Definition \ref{defn:nonbpsattractor}) has critical points a  sub-symmetric space $\mc{N}$ associated to the group $G_{\gamma}\defeq \textrm{Stab}_{\pm}(\gamma)$, the stabilizer of the class $\gamma$ up to sign:
\[
\rm{Stab}_{\pm}(\gamma)\defeq \{ g\in G|g\cdot \gamma=\pm \gamma\}.
\]

\end{thm}
 
\begin{proof}
We first consider charge vectors of the form $\gamma=(p^0, 0, q_0, 0)$, with $p^0q_0\neq 0$. In this case the black hole potential simplifies to \cite[Equation (2.13)]{ferrara4d5d}
\[
2V=\bigg[ \frac{\kappa}{6}(1+4g)+\frac{h^2}{6\kappa}+\frac{3}{8\kappa}g^{ij}h_ih_j\bigg](p^0)^2+\frac{6}{\kappa}(q_0)^2+\frac{2}{\kappa}hp^0q_0.
\]
Here, we have written $z^i=x^i-i\lambda^i$, and 
\begin{align*}
    &\kappa=6\lambda^i\lambda^j\lambda^k=\mc{V},\\
    &\kappa_{ij}=d_{ijk}\lambda^k, \ \kappa^{ij}\kappa_{jl}=\delta^i_l, \\ 
    &\kappahat_{ij}=d_{ijk}\lambdahat^k,\ \kappahat_i=d_{ijk}\lambdahat^j\lambdahat^k, \ \kappahat=6\lambdahat^i\lambdahat^j\lambdahat^k=6,\\ 
    & g_{ij}=\frac{1}{4}\big(\frac{1}{4}\kappahat_i\kappahat_j-\kappahat_{ij}\big)\mc{V}^{-2/3}, \ g_i=-4g_{ij}x^j, \ g=g_{ij}x^ix^j,\\
    & g^{ij}=2\big(\lambda^i\lambda^j-\frac{\kappa}{3}\kappa^{ij}\big),\\
    & h_i=d_{ijk}x^jx^k, \ h=d_{ijk}x^ix^jx^k.
\end{align*}
We will only need some particular features of the potential $V$ in this case, which we now describe. First observe that there are no terms which are linear in the  $x^i$'s, and hence along the locus 
\[
\mc{C}'=\{(x^i, \lambda^i)| x^i=0\ \text{for}\  i=1, \cdots , n\}
\]
the partial derivatives $\frac{\partial V}{\partial x^i}$ all vanish. On the other hand, along the locus $\mc{C}'$, the variables $\lambda^i$ only show up through the single function $\kappa$, and denoting by $V^*$ the restriction of $V$ along $\mc{C}'$, we have
\[
V^*(\kappa)=\frac{\kappa}{6}(p^0)^2+\frac{6}{\kappa}(q_0)^2.
\]
By differentiating $V^*$ with respect to $\kappa$, the attractor equation is therefore
\[
\frac{1}{6}(p^0)^2-\frac{6}{\kappa^2}(q_0)^2=0,
\]
and hence we have that the set 
\[
\mc{C}=\{(x^i, \lambda^i)\in \mc{C}'| \kappa=\frac{6q_0}{p^0}\}
\]
is contained in the critical points of $V$; in particular, the set of critical points is non-empty. Furthermore, it is in fact a connected component of the critical locus, since an easy computation shows that the Hessian of $V$ in the directions $\kappa, x^i$ (recall $V$ is a function of $\kappa, x^i$ since all the dependence through the $\lambda^i$'s are in fact through $\kappa$), when  evaluated on $\mc{C}'$, is given by the matrix  
\[
\begin{pmatrix}
12(q_0)^2/\kappa^3 & 0\\
0                       & (g_{ij})
\end{pmatrix},
\]
which is positive definite (note that $q_0\neq 0$ by our assumption at the beginning). By equivariance, we have that $\rm{Stab}(\gamma)(\mb{R})$ preserves the critical points of $V$.

Recall that we have the representation 
\[
V=\mb{Q}\oplus J\oplus J^{\vee}\oplus \mb{Q}^{\vee},
\]
and our charge vector lives in $\mb{Q}\oplus \mb{Q}^{\vee}$. There is an action of $G_5$ on $\mb{Q}\oplus J$ where the action on the first factor is the trivial one, and that on $J$ is the natural one. By taking this representation plus its dual,  $G_5$ also acts on $V$, and in fact this is the (real points of the) stabilizer of $\gamma$. We have that $\mc{C}'$ is an orbit of $G_5$: indeed, this is just the fact that the hypersurfaces 
\[
\{x\in J(\mb{R})|\kappa=\text{constant}\}
\]
are orbits the $G_5$-action on $J(\mb{R})$. Note also that by the equivariance of the $G$-action, 
any component of the critical points of $V$ will also be an orbit of $G_5$.

Now suppose  that $\mc{C}_2$ is another component in the critical locus. Take an element $g\in G$ which takes $\mc{C}_2$ to $\mc{C}$; our goal is to show that $g\in \Stab_{\pm}(\gamma)$.  By equivariance the potential $V_{g\cdot \gamma}$ has $\mc{C}$ as a connected component of its critical points. By Lemma \ref{lemma:samecriticalpoints} below, we have that the vector $\gamma'=(P^0, P^i, Q_0, Q_i)$ satisfies 
\begin{align}
    \begin{split}
        P^0P^i&=0,\\
        Q_iP^j&=0\\
        Q_0Q_i&=0
    \end{split}
\end{align}
for all choices of $i, j$. There are now  three possibilities for the form of $\gamma'$
\begin{enumerate}
    \item $\gamma'=(P^0, 0, Q_0, 0)$, or
    \item $\gamma'=(P^0, 0, 0, Q_i)$, or 
    \item $\gamma'=(0, P^i, Q_0, 0)$.
\end{enumerate}
We claim that only possibility (1) is possible. Indeed, if we are in cases (2) or (3), the potential $V^*=V|_{\mc{C}'}$ will not be simply a function of $\kappa$, and then the set of critical points cannot be an entire hypersurface $\kappa=\text{constant}$. Therefore we have $g\cdot \gamma = \gamma'=(P^0, 0, Q_0, 0)$ for some $P^0, Q_0$. 

Now the value of $\kappa$ is the same for $\gamma$ and $\gamma'$, and hence we have 
\[
\frac{p^0}{q_0}=\frac{P^0}{Q_0}.
\]
Therefore $\gamma$ and $\gamma'$ differ by a scalar, in other words, $g$ acts on $\gamma$ by a scalar.   Finally, since $g$ is an element of the group $G$ which preserves the quartic form, which is homogeneous in the $p^i$'s and $q_i$'s, we must have this scalar being $\pm$. Hence the critical points of $V$ is an orbit of the group $\rm{Stab}_{\pm}(\gamma)(\mb{R})$, and we have proven the theorem for charge vectors of the form $\gamma=(p^0, 0, q_0, 0)$. 

Finally, by \cite[Equation (2.21)]{ferrara4d5d}, we have 
\[
I_4(p^0, 0, q_0, 0)=-(p^0q_0)^2,
\]
where $I_4$ is the quartic invariant, and so is non-zero by our assumption. By transitivity of the action of $G$ and equivariance (since we may take $p^0, q_0$ to be arbitrary), this implies the theorem for all non-BPS charge vectors.
\end{proof} 

\begin{lem}\label{lemma:samecriticalpoints}
If a charge vector $(p^0, p^i, q_0, q_i)$ is such that $V_{\gamma}$ satisfies $\frac{\partial V}{\partial x^i}=0$ on the locus $\mc{C}$ for a charge vector of the form $\gamma'= (P^0, 0, Q_0, 0)$, then we must have 
\begin{align*}
    p^0p^i&=0\\
    q_ip^j&=0\\
    q_0q_i&=0
\end{align*}
for all choices of $i, j$.
\end{lem}

\begin{proof}
The terms of the potential $V$ linear in the $x^i$'s are given by 
\begin{equation}\label{eqn:linearterms}
\frac{\kappa}{6}p^0p^i(-4g_{ij}x^j)+\bigg(\lambda^i\lambda^m-\frac{\kappa}{3}\kappa^{im}\bigg)d_{mjk}q_ip^jx^k+\frac{12}{\kappa}x^iq_0q_i.
\end{equation}
The condition that $\frac{\partial V}{\partial x^i}$ vanishes  on the locus $\mc{C}$ is equivalent to (\ref{eqn:linearterms}) vanishing identically as a function of the $x^i$'s and $\lambda^i$'s, subject to the condition that $\kappa$ is a constant (which is determined by $\gamma'=(P^0, 0, Q_0, 0)$. First of all, as functions of the $\lambda^i$'s, the first summand of (\ref{eqn:linearterms}) consists of linear and quartic functions, which do not appear in the other two summands of (\ref{eqn:linearterms}): therefore the first summand vanishes identically. Similarly the third summand is constant as a function of the $\lambda^i$'s, and so must vanish identically as well. In other words, the three summands in (\ref{eqn:linearterms})   all vanish identically.

This immediately implies that $q_0q_i$ vanishes for all $i$, and furthermore, since the matrix $g_{ij}$ is invertible, we also have $p^0p^i$ being zero for all $i$.

By a similar argument as above, for the second summand of (\ref{eqn:linearterms}), we have 
\[
d_{mjk}x^kq_ip^j
\]
being identically zero, and hence 
\[
q_i(d_{mjk}p^j)=0.
\]
If $q_i\neq 0$, then $d_{mjk}p^j=0$, and hence 
\[
d_{mjk}p^jp^k=0.
\]
But the cubic norm given by $d_{ijk}$ is non-degenerate, and therefore $p^j=0$ for all $j$. This implies that $q_ip^j=0$ for all $i, j$, as required. 

\end{proof}

\begin{example}
For the case $J=\mb{R}$, the sub-symmetric space $\mc{N}$ is still $0$-dimensional and they are the points obtained  in Section \ref{section:nonbps}. 
\end{example}
\begin{example}\label{example:nonbpsherm3}
For the case $J=\herm_3(\mb{R})$, the space $\mc{N}$ is a copy of  the symmetric space $\SL(3, \mb{R})/\SO(3)$ embedded in $\Sp(6, \mb{R})/U(1)\times \SU(3)$. However the $\mb{Q}$-algebraic group defining $\mc{N}$ will depend on the class $\gamma$.
\end{example}

\begin{rmk}
We believe these sub-symmetric spaces are analogues of those constructed by Goresky-Tai in \cite{goreskytai} and \cite{goreskytaimultiplication}. In fact, we believe that the spaces we get in the case of Example  \ref{example:nonbpsherm3} are precisely the ones obtained in \cite{goreskytai} in the case $n=3$ (using  their notation).
\end{rmk}



\section{Other avatars of the attractor mechanism}\label{section:otheravatars}
The attractor mechanism is an extremely general phenomenon whose scope lies beyond just   moduli spaces of Calabi-Yau manifolds. In Section \ref{section:CYVHS} we described a mechanism which picks out points in Hermitian symmetric domains underlying Calabi-Yau variations of Hodge structures; in this section we describe an analogous attractor mechanism which picks out points in certain locally symmetric spaces which do not have Hermitian structure. It seems that these spaces should  be thought of as moduli spaces of Calabi-Yaus with real structure, analogous to the way in which the locally symmetric space for $\GL_n(\mb{R})$ sits inside the locally symmetric space for $\GSp(2n, \mb{R})$ (the Siegel modular variety) as the moduli space of tori with real structure: see \cite{goreskytai}. We hope to return to this point in future work. The spaces considered in this section are in one-one correspondence with those studied in Section \ref{section:CYVHS}.

\subsection{Real  symmetric spaces associated to Jordan algebras}

In this section we review the symmetry groups and symmetric spaces attached to degree 3 Jordan algebras. To this end let $J$ be one of the Jordan algebras considered in Section \ref{section:CYVHS}, and recall that we have a cubic norm on $J(\mb{R})$ given by 
\begin{equation}\label{eqn:cubicnorm}
    \mc{V}(h)=\sum_{I,J,K}d_{IJK}h^Ih^Jh^K
\end{equation}

where we have chosen a basis of $J$ and the $h^I$'s are coordinates on the vector space $J(\mb{R})$. 
\begin{defn}\hfill
\begin{enumerate}
    \item For each of the Jordan algebras, let $G$ denote the subgroup of $\GL(J)$ preserving the cubic norm $\mc{V}$ up to scalar, $G^0$ the subgroup preserving the cubic norm on the nose, and $K$ the stabilizer of some fixed vector.
    \item Let $\mc{M}$ denote the symmetric space $G^0(\mb{R})/K(\mb{R})$.
\end{enumerate}
\end{defn}


These have been previously studied in \cite{ferraragunaydin5d}, and the list of groups can be found in loc.cit. 
\begin{prop}\label{prop:bpsorbit}
On the hypersurface $\mc{V}=1$, there is a unique $G^0(\mb{R})$-orbit such that the stabilizer of a point is compact.
\end{prop}
\begin{proof}
This is easy to see for the theories with groups $\SL(2, \mb{R})\times \SO(1, n)$. For the exceptional theories, see \cite[Section 3]{ferraragunaydin5d}.
\end{proof}
\begin{rmk}\label{rmk:hypersurfacemoduli}
We will sometimes view $\mc{M}$ as the orbit specified by Proposition \ref{prop:bpsorbit} sitting as a hypersurface  inside $J(\mb{R})$. Note that sometimes $\mc{M}$ is simply claimed to be the hypersurface $\mc{V}=1$ in the physics literature, whereas one should  take only the component specified above. 
\end{rmk}


\begin{example}
We will make everything  explicit in the example of $J=\herm_3(\mb{R})$. In this case we have $G=\GL(3, \mb{R})$, where the action is given as follows. For $A\in \GL(3, \mb{R})$ and $M\in \herm_3(\mb{R})$ (note that in this case hermitian matrices simply means symmetric matrices) the action is
\[
A\cdot M=AMA^T.
\]
Equivalently this is the action of $\GL(3, \mb{R})$ on the space of quadratic forms. It is now clear that the stabilizer of the identity matrix in $J$ is $\O(3)$. Also the cubic norm $\mc{V}(h)$ is simply the determinant of the matrix $h$ in this case.

\end{example}

\subsection{Metric and potential on moduli space}
In order to define an attractor mechanism, we need a potential on the moduli space for each choice of charge vector $q\in V(\mb{Z})$. We will first give a formula for the metric on the symmetric space $G^0/K$, and then define this potential function. 

\begin{defn}
Define the metric  tensor on $V$ by 
\[
a_{IJ}\defeq \frac{-1}{3}\frac{\partial}{\partial h^I}\frac{\partial}{\partial h^J}\ln(\mc{V}(h)),
\]
where $\mc{V}$ is the cubic norm (\ref{eqn:cubicnorm}).
Then the natural metric on the symmetric space $\mc{M}=\{h| \mc{V}(h)=1\}$ is given by restricting $a_{IJ}$ onto the hypersurface $\mc{V}=1$:
\[
\mathring{a}_{IJ}\defeq a_{IJ}|_{\mc{V}=1}.
\]
We will also denote by $\mathring{a}^{IJ}$ the inverse matrix of $\mathring{a}_{IJ}$.
\end{defn}

We can now define the potential which will give us attractor points.
\begin{defn}
For a vector $q=(q_I)\in V(\mb{Z})$, define the potential
\[
V(q)\defeq \sum_{I,J}\mathring{a}^{IJ}q_Iq_J.
\]
This is a real valued  function on $\mc{M}$, and  we define attractor points to be  the critical points of $V$ with charge vector  $q_I$.
\end{defn}
As in the $d=4$ case these attractor points can be either BPS or non-BPS. We now spell out this distinction which is completely analogous to the $d=4$ situation. We have a central charge 
\begin{equation}\label{eqn:5dZ}
Z\defeq \sum_I h^Iq_I,
\end{equation}
and letting 
\begin{equation}
    g_{xy}\defeq \frac{3}{2} h_{I,x}h_{J,x}\mathring{a}^{IJ}
\end{equation}
\begin{equation}
    g^{xy}=(g_{xy})^{-1}
\end{equation}
(here the $x,y$ subscripts denote partial derivatives) a straightforward computation shows that 
\begin{equation}\label{eqn:5dpotential}
   V(q)=Z^2+\frac{3}{2}\sum_{x,y}g^{xy}\partial_xZ\partial_yZ, 
\end{equation}
which is the analogue of the formula for $V_{eff}$ in Definition \ref{defn:nonbpsattractor}.

\begin{defn}
As in the $d=4$ case, if an attractor point  minimizes the two terms in equation (\ref{eqn:5dpotential}) separately, then it is BPS; otherwise it is non-BPS.
\end{defn}
\begin{rmk}
In the symmetric cases, the BPS condition is also equivalent to the stabilizer of $q$ being a compact subgroup of $G^0$.
\end{rmk}

\subsection{BPS attractors in $5d$}
In this section we give a description of the BPS attractors in $5d$. Note that there are again non-BPS attractors,   and within the non-BPS class there are distinct classes, just as in the $d=4$ case; we will not consider these here.
Recall that in the setting of $5d$ attractors the moduli space is only a real symmetric space, which we view as a hypersurface inside $J\otimes \mb{R}$ (see Remark~\ref{rmk:hypersurfacemoduli}), and the charge vector is an element of $J(\mb{Q})$. Our main result is
\begin{lem}
For a charge vector $q\in J(\mb{Q})$, the attractor point is a point $x\in \mc{M}_5$ such that its tangent plane has slope $q_I$: that is, the equation of the tangent plane at $x$ is given by 
\[
\sum_{I}q_Ih^I=0.
\]
In other words, the attractor points correspond to ''integral points'' on the dual hypersurface inside $\mb{P}(J(\mb{R})^{\vee})$, where by an  integral point we mean the intersection of the dual hypersurface with a ray through an integral point in $J(\mb{R})^{\vee}$.
\end{lem}
\begin{proof}
This follows essentially from the definition of a BPS attractor. Let $x^i$ be coordinates on the hypersurface $\mc{V}=1$, where $i=1, \cdots , n-1$. Now a BPS attractor is a critical point of $|Z|^2$, and since $Z\neq 0$ at such a point, we have 
\begin{equation}\label{eqn:tangent}
\partial_{x_i}Z=0, \ i=1, \cdots, n-1.
\end{equation}
But $Z=\sum_{I}q_Ih^I$ is a linear form on the vector space $J(\mb{R})$, and so equation (\ref{eqn:tangent}) implies that the tangent vectors $\partial_{x_i}$ to the hypersurface $\mc{V}=1$ at an attractor point lie on the hyperplane $Z=0$, as required.
\end{proof}

\begin{rmk}
These points  seem to be a cubic analogue of the points considered in the classical Linnik problem. Indeed, in the classical case, Linnik considers the unit sphere $S^2=\{(x,y,z)\in \mb{R}^3|x^2+y^2+z^2=1$, and the radial projection of the integral points 
\[
V_d(\mb{Z})\defeq \{(a,b,c)\in \mb{Z}^3|a^2+b^2+c^2=d\}
\]
onto $S^2$, and asks whether these points equidistribute as $d\rightarrow \infty$. \footnote{One may also consider the indefinite analogue of this, namely by taking instead $S=\{(x,y,z)\in \mb{R}^3|x^2-y^2z^2=1$, which is actually closer in spirit to our spaces since they are non-compact} In our case, we consider the (cubic)  hypersurface
\[
\mc{V}=\{(h^i)\in \mb{R}^n|C(h^i)=1\},
\]
(or rather the dual hypersurface) and consider the radial projection of integral points again. One of the proofs of equidistribution, due to Duke \cite{duke}, is by using a theta lift and proving estimates for Fourier coefficients of half integral weight modular forms. We therefore find it interesting to ask:

\begin{question}
Is there an analogous cubic theta lift in our cases?
\end{question}

\end{rmk}
\hfill

\subsection{Evidence for agreement with Goresky-Tai}
cf Ferrara et al 4d/5d correspondence and entropies: some special cases 
We give some evidence that the non-BPS attractor flow picks out certain real symmetric subspaces of moduli space. Recall that if the complex dimension of the moduli space is $n$, then the non-BPS $Z\neq 0$ attractor moduli spaces have real dimension $n-1$. 

Suppose the $(2n+2)$-charge vector is given by 
\[
Q=(p^0, p^i, q_0, q_i).
\]
We will show that, for special choices of $Q$, the non-BPS attractor moduli space is an $n-1$ dimensional subspace of an $n$-dimensional space defined by Goresky-Tai in their study of  abelian varieties with real structures. Let us first recall some notation for the coordinates and the black hole potential well adapted for our situation. We follow the notation in \cite{ferrara4d5d}. The $5d$ moduli space can be parametrized as the hypersurface
\[
\frac{1}{3!}d_{ijk}\lambdahat^i\lambdahat^j\lambdahat^k=1
\]
inside the real vector space $J\otimes \mb{R}$; here $i,j,k=1, \cdots , n$, and $d_{ijk}$ is the cubic norma structure underlying the Jordan algebra $J$. The $4d$ moduli space can be parametrized by the ``upper half plane coordinates'' 
\[
z^i=x^i-i\lambda^i,
\]
where again $i=1, \cdots, n$, and the  $x^i, \lambda^i$ are real, and $\mc{V}(\lambda^i)>0$. We will often use $\mc{V}, \lambdahat^i$ instead of $\lambda_i$, where 
\begin{align*}
    \mc{V}&\defeq \frac{1}{3!}d_{ijk},\\
    \lambdahat^i&\defeq \mc{V}^{-\frac{1}{3}}\lambda^i.
\end{align*}
Note that the black hole potential can be written explicitly as (cf \cite[Equation (2.13)]{ferrara4d5d}):
\begin{align}\label{eqn:potentialexplicit}
\begin{split}
    &2V=\Big[ \frac{\kappa}{6}(1+4g)+\frac{h^2}{6\kappa}+\frac{3}{8\kappa}g^{ij}h_ih_j\Big](p^0)^2+\\
    &\Big[\frac{2}{3}\kappa g_{ij}+\frac{3}{2\kappa}(h_ih_j+h_{im}g^{mn}h_{nj})\Big]p^ip^j+\\
    &\frac{6}{\kappa}\Big[(q_0)^2+2x^iq_0q_i +(x^ix^j+\frac{1}{4}g^{ij})q_iq_j\Big]+\\
    &2\Big[\frac{\kappa}{6}g_i-\frac{h}{2\kappa}h_i-\frac{3}{4\kappa}g^{jm}h_mh_{ij}\Big]+\\
    &-\frac{2}{\kappa}\Big[-hp^0q_0+3q_0p^ih_i-(hx^i+\frac{3}{4}g^{ij}h_j)p^0q_i+3(h_jx^i+\frac{1}{2}g^{im}h_{mj})q_ip^j\Big],
    \end{split}
\end{align}
where 
\begin{align*}
    &\kappa_{ij}=d_{ijk}\lambda^k,\ \kappa_i=d_{ijk}\lambda^j\lambda^k, \ \kappa=6\lambda^i\lambda^j\lambda^k=\mc{V}, \ 
    \kappa^{ij}\kappa_{jl}=\delta^i_l,\\
    &\kappahat_{ij}=d_{ijk}\lambdahat^k,\ \kappahat_i=d-{ijk}\lambdahat^j\lambdahat^k, \ \kappahat=6\lambdahat^i\lambdahat^j\lambdahat^k=6,\\ 
    & g_{ij}=\frac{1}{4}\big(\frac{1}{4}\kappahat_i\kappahat_j-\kappahat_{ij}\big)\mc{V}^{-2/3}, \ g_i=-4g_{ij}x^j, \ g=g_{ij}x^ix^j,\\
    & g^{ij}=2\big(\lambda^i\lambda^j-\frac{\kappa}{3}\kappa^{ij}\big),\\
    & h_{ij}=d_{ijk}x^k, \ h_i=d_{ijk}x^jx^k, \ h=d_{ijk}x^ix^jx^k.
\end{align*}

The expression (\ref{eqn:potentialexplicit}) for $V$ may seem unwieldy, but we can make the following observations. We can split the variables up into the real and imaginary parts  $x^i$ and $\lambda^i$ repsectively, and then the only terms with dependence on the $x^i$'s are the following:
\[
x^i, \ g_i, \ g, \ h_{ij}, \ h_i, h.
\]
Among these, the only terms with linear dependence on the $x^i$'s are 
\[
x^i, \ g_i, \ h_{ij},
\]
and the coefficients of these vanish if 
\[
q_0q_i=p^0p^i=p^iq_j=0
\]
for all $i, j=1, \cdots, n$. Therefore the partial derivatives of $V$ with respect to the $x^i$'s certainly vanish along the subspace 
\[
\mc{C}=\{z^i|x^i=0 \ \forall i\}.
\]
We will be concerned with two types of such vectors: $Q=(p^0, 0, q_0, 0)$, and $Q=(p^0, 0, q_i, 0)$. For later use we record the value of the quartic form $I_4$ of these charge vectors in terms of the cubic form $I_3$:
\begin{prop}
For $Q=(p^0, 0, q_0, 0)$, the quartic form is given by
\[
I_4(Q)=-(p^0q_0)^2
\]
and for $Q=(p^0, 0, q_i, 0)$ the quartic form is 
\[
I_4(Q)=-4 p^0 I_3(q).
\]

\end{prop}
\begin{proof}
This is immediate from the formula for $I_4$.
\end{proof}

If we now further specialize to the charge vector 
\[
Q=(p^0, 0, q_0, 0),
\]
then the potential simplifies to  
\[
V=\Big[\frac{\kappa}{6}(1+4g)+\frac{h^2}{6\kappa}+\frac{3}{8\kappa}g^{ij}h_ih_j\Big](p^0)^2+\frac{6}{\kappa}(q_0)^2+\frac{2}{\kappa}hp^0q_0.
\]
As discussed above, the partials with respect to the $x_i$'s vanish automatically by our choice of charge vector, and furthermore along the subsapce $\mc{C}$ the potential reduces further to 
\[
V=\frac{\kappa}{6}(p^0)^2+\frac{6}{\kappa}(q_0)^2,
\]
which is a function of $\kappa$ only. Thus we have the following
\begin{prop}
If the charge vector is $Q=(p^0,0, q_0, 0)$, then the attractor moduli space is defined  by the equations 
\begin{align*}
x^i=0, \ \mc{V}&=\frac{q_0}{p^0}.
\end{align*}
\end{prop}

\begin{proof}
Recall that we have coordinates $x^i, \lambdahat^i, \mc{V}$ for the moduli space $\mc{M}$. Now the dimension of the attractor moduli space is $n-1$, and using the equality $\kappa=6\mc{V}$ and differentiating $V$ with respect to $\mc{V}$ yields
\[
\mc{V}=\bigg|\frac{q_0}{p^0}\bigg|,
\]
since $\mc{V}$ is assumed to be positive. Therefore we have found a subspace of dimension $n-1$, namely
\[
\mc{C}_Q\defeq \{(z^i)|x^i=0 \ \mc{V}=\frac{q_0}{p^0}\}.
\]
So this must be the attractor moduli space for the charge vector $Q$.
\end{proof}

We now compare these subspaces with those defined by Goresky-Tai \cite{goreskytai}, in the case of $G=\Sp(6, \mb{R})$. We find that they agree, except that our space is a symmetric space for $\SL(3, \mb{R})$ whereas that in loc.cit. is a symmetric space for $\GL(3, \mb{R})$. Therefore, in some sense, the non-BPS attractor moduli spaces are a refinement of the construction given in loc.cit.
\begin{prop}
The subspaces $\mc{C}_Q$ are codimension 1 subspaces inside the moduli of real abelian threefolds defined in \cite{goreskytai}. More precisely, each  $\mc{C}_Q$ is contained in a connected component of the image of  $X(4)(\mb{R})$ inside $X(1)$. Here $X(N)$ denotes the moduli space of principally polarized abelian threefolds with level $N$ structure, i.e. $\mf{h}_3/\Gamma(N)$.
\end{prop}
\begin{rmk}
This also meshes well with the expectation, coming from physics,  that the non-BPS $Z\neq 0$ attractor moduli spaces are isomorphic to the $5d$-moduli spaces, which are in turn supposed to be related to moduli of real varieties 
\end{rmk}

\begin{proof}
Consider the inverse image of $X(4)(\mb{R})$ inside the Siegel upper half space $\mf{h}_3$. By \cite[Section 6.4]{goreskytai} this consists of the set of points $Z$ for which there exists $\gamma\in \Gamma(4)$ such that 
\[
\gamma \cdot Z=-\bar{Z},
\]
where $\bar{Z}$ denotes the complex conjugate of $Z$. Now the points in $\mc{C}_Q$ are such that $x^i=0$ for all $i$, and therefore $X=0$, and therefore satisfy $Z=-\bar{Z}$, as required.
\end{proof}

We also note that the non-BPS $Z\neq 0$ attractor moduli spaces contain CM points; more precisely we have the following 
\begin{prop}
Let $(q_i)$ be a $5d$ charge vector  sastifying  $I_3(q)>0$, i.e. that it corresponds to a BPS attractor point in $5d$. Then for $p^0<0$, 
\begin{itemize}
    \item the charge vector $Q=(p^0, 0, 0, q_i)$ gives rise to a BPS attractor point, whereas
    \item $Q=(-p^0, 0,0, q_i)$ gives rise to a non-BPS $Z\neq 0$ attractor moduli space.
\end{itemize}
Furthermore the former is contained in  the latter.
\end{prop}
\begin{proof}
The proof is essentially a computation in \cite{ferrara4d5d}. Again restricting ourselves to the slice $x^i=0$, we have that the partial derivatives of $V$ with respect to the $x^i$'s vanish automatically. Then for a charge vector $Q=(p^0, 0, 0, q_i)$ the potential (\ref{eqn:potentialexplicit}) simplifies to 
\[
V=\frac{1}{2}\Big[(p^0)^2\mc{V}+\mc{V}^{-\frac{1}{3}}a^{ij}q_iq_j\Big]
\]
Now we observe that this is independent of the sign of $p^0$, and therefore if a solution exists then it is both a (BPS) attractor point and a point of the non-BPS moduli space.  It remains to show that a solution exists. But this is equivalent to showing that the attractor point in $5d$ for the charge vector $(q_i)$ exists, which is indeed the case 

\end{proof}

\begin{rmk}
Note that this will not be true in general, since the non-BPS attractor points in the $t^3$-model, as computed in Section \ref{section:nonbps}, shows.
\end{rmk}

\section{Further questions}\label{section:further}
\subsection{Geometric realizations}
As mentioned before, the question of finding geometric realizations of the variation of Hodge structures studied here is an important one. There are actually two different forms of this question: the A-model side and the B-model side, and they are related by mirror symmetry. The B-model question is what we have been discussing, namely that of finding a family of algebraic varieties such that the variations of Hodge structures appear in  cohomology (or subquotients thereof). For the A-model question, we mean finding a family of Calabi-Yau's such that the intersection product on $H^2$ is given by the Jordan structure constants $N_{ijk}$: the question in this form seems to have been first raised by Looijenga-Lunts \cite[Section 4]{looijengalunts}.

Here we switch our attention to the A-model. In particular we consider the case of $J=\herm_3(\mb{C})$. Physicists have argued that this is the A-model on the orbifold $T^6/(\mb{Z}/3)$. Note that this is singular with 27 orbifold points, and upon blowing up at these 27 points the Hodge diamond is as follows.

\begin{equation}
\begin{array}{ccccccc}
&&& 1 &&&\\
&& 0 && 0 &&\\
& 0 && 36 && 0 &\\ 
1 && 0 && 0 && 1\\
& 0 && 36 && 0 &\\ 
&& 0 && 0 &&\\
&&& 1 &&&.\\
\end{array}
\end{equation}

This manifold is referred to as the $Z$-manifold nowadays, and was the first example of a rigid CY for which mirror symmetry was considered. The mirror of this is a (resolution of a) quotient of a cubic sevenfold, and it was observed \cite{} that the middle Hodge numbers of this family of  sevenfolds are $(0, 0, 1, 36, 36, 1, 0,0)$, precisely the Hodge numbers required of a mirror.  On the other hand it has been argued that if we consider the initial orbifold with $h^{11}=9$, the moduli space should exactly the symmetric space for $\SU(3,3)$. Therefore this suggests that

\begin{question} 
Does there exist a 9-dimensional sublocus of this 36 family of sevenfolds such that the (middle degree) Hodge structure splits, and such that the  Hodge structure   for $\herm_3(\mb{C})$ shows up as a piece?
\end{question}

We should also note that it is possible to construct the family of Hodge structures as the third exterior power of $H^1$ of an abelian of Weil-type. Therefore if the above question has an affirmative answer there should be a geometric relationship between this family of Weil-type abelian varieties and the sublocus of cubic sevenfolds, as predicted by the Hodge conjecture.

We remark also that certain automorphic forms on these moduli spaces have been constructed, which are supposed to be analogues of the $\eta$ function for the usual upper half plane: see \cite{dualityinvariantpartition}. In the case of $\SU(3,3)$ we can also ask if this function has to do with the analytic torsion of the mirror family.



\subsection{Multi-centered black holes and elliptic curves}
In this work the attractor points correspond to single centered black holes; there is an intriguing relation between multi-centered black holes in the theories we consider and elliptic curves found by L\'evay \cite{levay}, which seems to come from arithmetic invariant theory. It would be interesting to investigate this relation further in our context: for example, the two centered black holes have two associated charge vectors, corresponding to two CM points, and it is tempting to see if these are related to the elliptic curves in loc.cit.







\subsection{Higher supersymmetry}
There are also variants of the attractor mechanism for supergravity theories  with more supersymmetry, whose moduli spaces are automatically symmetric. Concretely this means that, for example, in the moduli spaces for the $4d$ theory, maximal compact subgroup of the symmetry group no longer has a $\rm{U}(1)$ factor, but instead a bigger group appears as a factor. For example, in the case $N=4$ the moduli spaces are given by the family $\SO(6, n)/\SO(6)\times \SO(n)$. We refer the reader to \cite[Table 6]{ferrarasymmetricspace} for details.
\printbibliography[
]  

@incollection {moore,
    AUTHOR = {Moore, Gregory},
     TITLE = {Strings and arithmetic},
 BOOKTITLE = {Frontiers in number theory, physics, and geometry. {II}},
     PAGES = {303--359},
 PUBLISHER = {Springer, Berlin},
      YEAR = {2007},
   MRCLASS = {81T30 (11F50 11Z05 14J28 58J26)},
  MRNUMBER = {2290765},
MRREVIEWER = {Marcel L. Vonk},
       DOI = {10.1007/978-3-540-30308-4_8},
   %    URL = {https://doi.org/10.1007/978-3-540-30308-4_8},
}

@article {gepner,
    AUTHOR = {Gepner, Doron},
     TITLE = {Exactly solvable string compactifications on manifolds of
              {${\rm SU}(N)$} holonomy},
   JOURNAL = {Phys. Lett. B},
  FJOURNAL = {Physics Letters. B. Particle Physics, Nuclear Physics and
              Cosmology},
    VOLUME = {199},
      YEAR = {1987},
    NUMBER = {3},
     PAGES = {380--388},
      ISSN = {0370-2693},
   MRCLASS = {83E15 (14J28 53C80 81E30 81E40 83E30)},
  MRNUMBER = {929596},
MRREVIEWER = {Daniel G. Caldi},
       DOI = {10.1016/0370-2693(87)90938-5},
   %    URL = {https://doi-org.ezp-prod1.hul.harvard.edu/10.1016/0370-2693(87)90938-5},
}

@article {gross,
    AUTHOR = {Gross, Benedict H.},
     TITLE = {A remark on tube domains},
   JOURNAL = {Math. Res. Lett.},
  FJOURNAL = {Mathematical Research Letters},
    VOLUME = {1},
      YEAR = {1994},
    NUMBER = {1},
     PAGES = {1--9},
      ISSN = {1073-2780},
   MRCLASS = {14D07 (14C30 32J25 32M15)},
  MRNUMBER = {1258484},
       DOI = {10.4310/MRL.1994.v1.n1.a1},
  %     URL = {https://doi-org.ezp-prod1.hul.harvard.edu/10.4310/MRL.1994.v1.n1.a1},
}

@article {heateqn,
    AUTHOR = {G\"{u}naydin, Murat and Neitzke, Andrew and Pioline, Boris},
     TITLE = {Topological wave functions and heat equations},
   JOURNAL = {J. High Energy Phys.},
  FJOURNAL = {Journal of High Energy Physics. A SISSA Journal},
      YEAR = {2006},
    NUMBER = {12},
     PAGES = {070, 41},
      ISSN = {1126-6708},
   MRCLASS = {81T45 (33E05 58J35 81T30)},
  MRNUMBER = {2276684},
MRREVIEWER = {Vincent Bouchard},
       DOI = {10.1088/1126-6708/2006/12/070},
   %    URL = {https://doi-org.ezp-prod1.hul.harvard.edu/10.1088/1126-6708/2006/12/070},
}

@article{stu,
    AUTHOR = {Bellucci, Stefano and Ferrara, Sergio and Marrani, Alessio and
              Yeranyan, Armen},
     TITLE = {{$stu$} black holes unveiled},
   JOURNAL = {Entropy},
  FJOURNAL = {Entropy. An International and Interdisciplinary Journal of
              Entropy and Information Studies},
    VOLUME = {10},
      YEAR = {2008},
    NUMBER = {4},
     PAGES = {507--553},
   MRCLASS = {83C57 (81T20 81T30)},
  MRNUMBER = {2465847},
MRREVIEWER = {Farhang Loran},
       DOI = {10.3390/e10040507},
   %    URL = {https://doi-org.ezp-prod1.hul.harvard.edu/10.3390/e10040507},
}

@article{kachruetc,
  title={Distributions of extremal black holes in Calabi-Yau compactifications},
  author={Hulsey, George and Kachru, Shamit and Yang, Sungyeon and Zimet, Max},
  journal={arXiv preprint arXiv:1901.10614},
  year={2019}
}

@article {denefdouglas,
    AUTHOR = {Denef, Frederik and Douglas, Michael R.},
     TITLE = {Distributions of nonsupersymmetric flux vacua},
   JOURNAL = {J. High Energy Phys.},
  FJOURNAL = {Journal of High Energy Physics. A SISSA Journal},
      YEAR = {2005},
    NUMBER = {3},
     PAGES = {061, 30},
      ISSN = {1126-6708},
   MRCLASS = {81T30},
  MRNUMBER = {2151879},
MRREVIEWER = {James D. Wells},
       DOI = {10.1088/1126-6708/2005/03/061},
    %   URL = {https://doi-org.ezp-prod1.hul.harvard.edu/10.1088/1126-6708/2005/03/061},
}

@article {hough,
    AUTHOR = {Hough, Bob},
     TITLE = {Equidistribution of bounded torsion {CM} points},
   JOURNAL = {J. Anal. Math.},
  FJOURNAL = {Journal d'Analyse Math\'{e}matique},
    VOLUME = {138},
      YEAR = {2019},
    NUMBER = {2},
     PAGES = {765--797},
      ISSN = {0021-7670},
   MRCLASS = {11H55 (11K36 11R11 13C20)},
  MRNUMBER = {3996057},
       DOI = {10.1007/s11854-019-0044-4},
   %    URL = {https://doi-org.ezp-prod1.hul.harvard.edu/10.1007/s11854-019-0044-4},
}

@article {DSZ1,
    AUTHOR = {Douglas, Michael R. and Shiffman, Bernard and Zelditch, Steve},
     TITLE = {Critical points and supersymmetric vacua. {I}},
   JOURNAL = {Comm. Math. Phys.},
  FJOURNAL = {Communications in Mathematical Physics},
    VOLUME = {252},
      YEAR = {2004},
    NUMBER = {1-3},
     PAGES = {325--358},
      ISSN = {0010-3616},
   MRCLASS = {32L81 (32L10 58J90 81T60)},
  MRNUMBER = {2104882},
MRREVIEWER = {David Borthwick},
       DOI = {10.1007/s00220-004-1228-y},
   %    URL = {https://doi-org.ezp-prod1.hul.harvard.edu/10.1007/s00220-004-1228-y},
}

@article {DSZII,
    AUTHOR = {Douglas, Michael R. and Shiffman, Bernard and Zelditch, Steve},
     TITLE = {Critical points and supersymmetric vacua. {II}. {A}symptotics
              and extremal metrics},
   JOURNAL = {J. Differential Geom.},
  FJOURNAL = {Journal of Differential Geometry},
    VOLUME = {72},
      YEAR = {2006},
    NUMBER = {3},
     PAGES = {381--427},
      ISSN = {0022-040X},
   MRCLASS = {32L81 (32L10 58J90 81T60)},
  MRNUMBER = {2219939},
MRREVIEWER = {David Borthwick},
   %    URL = {http://projecteuclid.org.ezp-prod1.hul.harvard.edu/euclid.jdg/1143593745},
}

@article {DSZIII,
    AUTHOR = {Douglas, Michael R. and Shiffman, Bernard and Zelditch, Steve},
     TITLE = {Critical points and supersymmetric vacua. {III}. {S}tring/{M}
              models},
   JOURNAL = {Comm. Math. Phys.},
  FJOURNAL = {Communications in Mathematical Physics},
    VOLUME = {265},
      YEAR = {2006},
    NUMBER = {3},
     PAGES = {617--671},
      ISSN = {0010-3616},
   MRCLASS = {32L81 (32Q25 58J50 58J90 81T30)},
  MRNUMBER = {2231684},
MRREVIEWER = {David Borthwick},
       DOI = {10.1007/s00220-006-0003-7},
   %    URL = {https://doi-org.ezp-prod1.hul.harvard.edu/10.1007/s00220-006-0003-7},
}

@book {julia,
    AUTHOR = {Julia, Gaston},
     TITLE = {\'{E}tude sur les formes binaires non quadratiques \`a ind\'{e}termin\'{e}es
              r\'{e}elles, ou complexes, ou \`a ind\'{e}termin\'{e}es conjugu\'{e}es},
 PUBLISHER = {NUMDAM, [place of publication not identified]},
      YEAR = {1917},
     PAGES = {293},
   MRCLASS = {Thesis},
  MRNUMBER = {3532882},
       URL = {http://www.numdam.org/item?id=THESE_1917__13__1_0},
}

@article {cremona,
    AUTHOR = {Cremona, J. E.},
     TITLE = {Reduction of binary cubic and quartic forms},
   JOURNAL = {LMS J. Comput. Math.},
  FJOURNAL = {LMS Journal of Computation and Mathematics},
    VOLUME = {2},
      YEAR = {1999},
     PAGES = {64--94},
   MRCLASS = {11E76 (11H55)},
  MRNUMBER = {1693411},
MRREVIEWER = {Renaud Coulangeon},
       DOI = {10.1112/S1461157000000073},
   %    URL = {https://doi-org.ezp-prod1.hul.harvard.edu/10.1112/S1461157000000073},
}

@article {cremonastoll,
    AUTHOR = {Stoll, Michael and Cremona, John E.},
     TITLE = {On the reduction theory of binary forms},
   JOURNAL = {J. Reine Angew. Math.},
  FJOURNAL = {Journal f\"{u}r die Reine und Angewandte Mathematik. [Crelle's
              Journal]},
    VOLUME = {565},
      YEAR = {2003},
     PAGES = {79--99},
      ISSN = {0075-4102},
   MRCLASS = {11H55 (11E76)},
  MRNUMBER = {2024647},
MRREVIEWER = {Christine Bachoc},
       DOI = {10.1515/crll.2003.106},
   %    URL = {https://doi-org.ezp-prod1.hul.harvard.edu/10.1515/crll.2003.106},
}

@article {fgk,
    AUTHOR = {Ferrara, Sergio and Gibbons, Gary W. and Kallosh, Renata},
     TITLE = {Black holes and critical points in moduli space},
   JOURNAL = {Nuclear Phys. B},
  FJOURNAL = {Nuclear Physics. B. Theoretical, Phenomenological, and
              Experimental High Energy Physics. Quantum Field Theory and
              Statistical Systems},
    VOLUME = {500},
      YEAR = {1997},
    NUMBER = {1-3},
     PAGES = {75--93},
      ISSN = {0550-3213},
   MRCLASS = {83C57 (83E50)},
  MRNUMBER = {1471650},
       DOI = {10.1016/S0550-3213(97)00324-6},
   %    URL = {https://doi-org.ezp-prod1.hul.harvard.edu/10.1016/S0550-3213(97)00324-6},
}

@article {goreskytai,
    AUTHOR = {Goresky, Mark and Tai, Yung Sheng},
     TITLE = {The moduli space of real abelian varieties with level
              structure},
   JOURNAL = {Compositio Math.},
  FJOURNAL = {Compositio Mathematica},
    VOLUME = {139},
      YEAR = {2003},
    NUMBER = {1},
     PAGES = {1--27},
      ISSN = {0010-437X},
   MRCLASS = {14K10 (11G10 14P05)},
  MRNUMBER = {2024963},
MRREVIEWER = {Jae-Hyun Yang},
       DOI = {10.1023/B:COMP.0000005079.56232.e3},
   %    URL = {https://doi-org.ezp-prod1.hul.harvard.edu/10.1023/B:COMP.0000005079.56232.e3},
}

@article {ferraramarrani,
    AUTHOR = {Ferrara, Sergio and Marrani, Alessio},
     TITLE = {{$\mathscr{N}=8$} non-{BPS} attractors, fixed scalars and magic
              supergravities},
   JOURNAL = {Nuclear Phys. B},
  FJOURNAL = {Nuclear Physics. B. Theoretical, Phenomenological, and
              Experimental High Energy Physics. Quantum Field Theory and
              Statistical Systems},
    VOLUME = {788},
      YEAR = {2008},
    NUMBER = {1-2},
     PAGES = {63--88},
      ISSN = {0550-3213},
   MRCLASS = {83E50 (53C80 83C57)},
  MRNUMBER = {2362631},
       DOI = {10.1016/j.nuclphysb.2007.07.028},
   %    URL = {https://doi-org.ezp-prod1.hul.harvard.edu/10.1016/j.nuclphysb.2007.07.028},
}

@article {goreskytaimultiplication,
    AUTHOR = {Goresky, Mark and Tai, Yung Sheng},
     TITLE = {Anti-holomorphic multiplication and a real algebraic modular
              variety},
   JOURNAL = {J. Differential Geom.},
  FJOURNAL = {Journal of Differential Geometry},
    VOLUME = {65},
      YEAR = {2003},
    NUMBER = {3},
     PAGES = {513--560},
      ISSN = {0022-040X},
   MRCLASS = {11G18 (11F46)},
  MRNUMBER = {2064430},
MRREVIEWER = {Mihran Papikian},
   %    URL = {http://projecteuclid.org.ezp-prod1.hul.harvard.edu/euclid.jdg/1434052758},
}

@article {thorne,
    AUTHOR = {Thorne, Jack A.},
     TITLE = {On the average number of 2-{S}elmer elements of elliptic
              curves over {$\Bbb F_q(X)$} with two marked points},
   JOURNAL = {Doc. Math.},
  FJOURNAL = {Documenta Mathematica},
    VOLUME = {24},
      YEAR = {2019},
     PAGES = {1179--1223},
      ISSN = {1431-0635},
   MRCLASS = {11G05 (14H52 14H60)},
  MRNUMBER = {4012556},
}

@article {kudlamillson,
    AUTHOR = {Kudla, Stephen S. and Millson, John J.},
     TITLE = {Intersection numbers of cycles on locally symmetric spaces and
              {F}ourier coefficients of holomorphic modular forms in several
              complex variables},
   JOURNAL = {Inst. Hautes \'{E}tudes Sci. Publ. Math.},
  FJOURNAL = {Institut des Hautes \'{E}tudes Scientifiques. Publications
              Math\'{e}matiques},
    NUMBER = {71},
      YEAR = {1990},
     PAGES = {121--172},
      ISSN = {0073-8301},
   MRCLASS = {11F32 (11F30 11F46 11F67 32N10 32N15)},
  MRNUMBER = {1079646},
       URL = {http://www.numdam.org/item?id=PMIHES_1990__71__121_0},
}

@article {ferraragunaydin5d,
    AUTHOR = {Ferrara, Sergio and G\"{u}naydin, Murat},
     TITLE = {Orbits and attractors for {$N=2$} {M}axwell-{E}instein
              supergravity theories in five dimensions},
   JOURNAL = {Nuclear Phys. B},
  FJOURNAL = {Nuclear Physics. B. Theoretical, Phenomenological, and
              Experimental High Energy Physics. Quantum Field Theory and
              Statistical Systems},
    VOLUME = {759},
      YEAR = {2006},
    NUMBER = {1-2},
     PAGES = {1--19},
      ISSN = {0550-3213},
   MRCLASS = {83E50 (17C90)},
  MRNUMBER = {2282384},
MRREVIEWER = {Luis Joaqu\'{\i}n Boya},
       DOI = {10.1016/j.nuclphysb.2006.09.016},
   %    URL = {https://doi-org.ezp-prod1.hul.harvard.edu/10.1016/j.nuclphysb.2006.09.016},
}

@article {ferrara4d5d,
    AUTHOR = {Ceresole, Anna and Ferrara, Sergio and Marrani, Alessio},
     TITLE = {4d/5d correspondence for the black hole potential and its
              critical points},
   JOURNAL = {Classical Quantum Gravity},
  FJOURNAL = {Classical and Quantum Gravity},
    VOLUME = {24},
      YEAR = {2007},
    NUMBER = {22},
     PAGES = {5651--5666},
      ISSN = {0264-9381},
   MRCLASS = {83C57 (83E15)},
  MRNUMBER = {2365148},
       DOI = {10.1088/0264-9381/24/22/023},
    %   URL = {https://doi-org.ezp-prod1.hul.harvard.edu/10.1088/0264-9381/24/22/023},
}

@article {shengzuo,
    AUTHOR = {Sheng, Mao and Zuo, Kang},
     TITLE = {Polarized variation of {H}odge structures of {C}alabi-{Y}au
              type and characteristic subvarieties over bounded symmetric
              domains},
   JOURNAL = {Math. Ann.},
  FJOURNAL = {Mathematische Annalen},
    VOLUME = {348},
      YEAR = {2010},
    NUMBER = {1},
     PAGES = {211--236},
      ISSN = {0025-5831},
   MRCLASS = {32G20 (14D07 14J32 32Q25)},
  MRNUMBER = {2657440},
MRREVIEWER = {Jan Nagel},
       DOI = {10.1007/s00208-009-0378-9},
    %   URL = {https://doi-org.ezp-prod1.hul.harvard.edu/10.1007/s00208-009-0378-9},
}

@incollection {hitchin,
    AUTHOR = {Hitchin, N. J.},
     TITLE = {The moduli space of complex {L}agrangian submanifolds},
      NOTE = {Sir Michael Atiyah: a great mathematician of the twentieth
              century},
   JOURNAL = {Asian J. Math.},
  FJOURNAL = {Asian Journal of Mathematics},
    VOLUME = {3},
      YEAR = {1999},
    NUMBER = {1},
     PAGES = {77--91},
      ISSN = {1093-6106},
   MRCLASS = {32G13 (32G07 53C26 53C56 53D12)},
  MRNUMBER = {1701923},
MRREVIEWER = {Steven B. Bradlow},
       DOI = {10.4310/AJM.1999.v3.n1.a4},
    %   URL = {https://doi-org.ezp-prod1.hul.harvard.edu/10.4310/AJM.1999.v3.n1.a4},
}

@article {cecottigeometry,
    AUTHOR = {Cecotti, S. and Ferrara, S. and Girardello, L.},
     TITLE = {Geometry of type {II} superstrings and the moduli of
              superconformal field theories},
   JOURNAL = {Internat. J. Modern Phys. A},
  FJOURNAL = {International Journal of Modern Physics A. Particles and
              Fields. Gravitation. Cosmology},
    VOLUME = {4},
      YEAR = {1989},
    NUMBER = {10},
     PAGES = {2475--2529},
      ISSN = {0217-751X},
   MRCLASS = {81E99 (81E30 81G20 83E30 83E50)},
  MRNUMBER = {1017548},
MRREVIEWER = {C. Gauthier},
       DOI = {10.1142/S0217751X89000972},
   %    URL = {https://doi-org.ezp-prod1.hul.harvard.edu/10.1142/S0217751X89000972},
}

@article{pollack,
  title={The Fourier expansion of modular forms on quaternionic exceptional groups},
  author={Pollack, Aaron},
  journal={Duke Mathematical Journal},
  year={2020},
  publisher={Duke University Press}
}

@article {krutelevich,
    AUTHOR = {Krutelevich, Sergei},
     TITLE = {Jordan algebras, exceptional groups, and {B}hargava
              composition},
   JOURNAL = {J. Algebra},
  FJOURNAL = {Journal of Algebra},
    VOLUME = {314},
      YEAR = {2007},
    NUMBER = {2},
     PAGES = {924--977},
      ISSN = {0021-8693},
   MRCLASS = {20G05 (11E76 17C99)},
  MRNUMBER = {2344592},
MRREVIEWER = {Pierre-Emmanuel Chaput},
       DOI = {10.1016/j.jalgebra.2007.02.060},
   %    URL = {https://doi-org.ezp-prod1.hul.harvard.edu/10.1016/j.jalgebra.2007.02.060},
}

@article {friedmanlaza,
    AUTHOR = {Friedman, Robert and Laza, Radu},
     TITLE = {Semialgebraic horizontal subvarieties of {C}alabi-{Y}au type},
   JOURNAL = {Duke Math. J.},
  FJOURNAL = {Duke Mathematical Journal},
    VOLUME = {162},
      YEAR = {2013},
    NUMBER = {12},
     PAGES = {2077--2148},
      ISSN = {0012-7094},
   MRCLASS = {14D07 (32G20 32M15)},
  MRNUMBER = {3102477},
MRREVIEWER = {Carlo Giovanni Madonna},
       DOI = {10.1215/00127094-2348107},
   %    URL = {https://doi-org.ezp-prod1.hul.harvard.edu/10.1215/00127094-2348107},
}

@incollection{gunaydinlectures,
  title={Lectures on spectrum generating symmetries and U-duality in supergravity, extremal black holes, quantum attractors and harmonic superspace},
  author={G{\"u}naydin, Murat},
  booktitle={The Attractor Mechanism},
  pages={31--84},
  year={2010},
  publisher={Springer}
}

@article{ferrarasymmetricspace,
  title={Symmetric spaces in supergravity},
  author={Ferrara, Sergio and Marrani, Alessio},
  year={2009}
}

@incollection {firstorderflow,
    AUTHOR = {Andrianopoli, L. and D'Auria, R. and Ferrara, S. and
              Trigiante, M.},
     TITLE = {Black holes and first order flows in supergravity},
 BOOKTITLE = {Supersymmetry in mathematics and physics},
    SERIES = {Lecture Notes in Math.},
    VOLUME = {2027},
     PAGES = {17--43},
 PUBLISHER = {Springer, Heidelberg},
      YEAR = {2011},
   MRCLASS = {83E50 (83C57)},
  MRNUMBER = {2885264},
MRREVIEWER = {Michele Cicoli},
       DOI = {10.1007/978-3-642-21744-9_2},
   %    URL = {https://doi-org.ezp-prod1.hul.harvard.edu/10.1007/978-3-642-21744-9_2},
}

@article {ngo,
    AUTHOR = {H{\fontencoding{T5}\selectfont \`\ocircumflex}, Q. P. and L\^{e} H\`ung, V. B. and Ng\^{o}, B. C.},
     TITLE = {Average size of 2-{S}elmer groups of elliptic curves over
              function fields},
   JOURNAL = {Math. Res. Lett.},
  FJOURNAL = {Mathematical Research Letters},
    VOLUME = {21},
      YEAR = {2014},
    NUMBER = {6},
     PAGES = {1305--1339},
      ISSN = {1073-2780},
   MRCLASS = {11G05 (14G25)},
  MRNUMBER = {3335849},
MRREVIEWER = {Jeroen Sijsling},
     %  DOI = {10.4310/MRL.2014.v21.n6.a6},
      % URL = {https://doi-org.ezp-prod1.hul.harvard.edu/10.4310/MRL.2014.v21.n6.a6},
}

@article{freed,
  title={Special K{\"a}hler manifolds},
  author={Freed, Daniel S},
  journal={Communications in Mathematical Physics},
  volume={203},
  number={1},
  pages={31--52},
  year={1999},
  publisher={Springer}
}

@article{ferraraobservations,
  title={Observations on the Darboux coordinates for rigid special geometry},
  author={Ferrara, Sergio and Maci{\'a}, {\'O}scar},
  journal={Journal of High Energy Physics},
  volume={2006},
  number={05},
  pages={008},
  year={2006},
  publisher={IOP Publishing}
}

@article {composite,
    AUTHOR = {Ferrara, Sergio and Gimon, Eric G. and Kallosh, Renata},
     TITLE = {Magic supergravities, {$N=8$} black hole composites},
   JOURNAL = {Phys. Rev. D (3)},
  FJOURNAL = {Physical Review. D. Third Series},
    VOLUME = {74},
      YEAR = {2006},
    NUMBER = {12},
     PAGES = {125018, 18},
      ISSN = {0556-2821},
   MRCLASS = {83E50 (81T30 83C57)},
  MRNUMBER = {2288154},
MRREVIEWER = {Manoel F. Borges},
       DOI = {10.1103/PhysRevD.74.125018},
      % URL = {https://doi-org.ezp-prod1.hul.harvard.edu/10.1103/PhysRevD.74.125018},
}

@article{chargeorbit,
  title={Charge orbits of symmetric special geometries and attractors},
  author={Bellucci, Stefano and Ferrara, Sergio and G{\"u}naydin, Murat and Marrani, Alessio},
  journal={International Journal of Modern Physics A},
  volume={21},
  number={25},
  pages={5043--5097},
  year={2006},
  publisher={World Scientific}
}

@article{helenius,
  title={Freudenthal triple systems by root system methods},
  author={Helenius, Fred W},
  journal={Journal of Algebra},
  volume={357},
  pages={116--137},
  year={2012},
  publisher={Elsevier}
}

@book {ggk,
    AUTHOR = {Green, Mark and Griffiths, Phillip and Kerr, Matt},
     TITLE = {Mumford-{T}ate groups and domains},
    SERIES = {Annals of Mathematics Studies},
    VOLUME = {183},
      NOTE = {Their geometry and arithmetic},
 PUBLISHER = {Princeton University Press, Princeton, NJ},
      YEAR = {2012},
     PAGES = {viii+289},
      ISBN = {978-0-691-15425-1},
   MRCLASS = {14C30 (32G20)},
  MRNUMBER = {2918237},
MRREVIEWER = {Christian Schnell},
}

@article{homogeneousspecial,
  title={Critical points of the Black-Hole potential for homogeneous special geometries},
  author={D'Auria, Riccardo and Trigiante, Mario and Ferrara, Sergio},
  journal={Journal of High Energy Physics},
  volume={2007},
  number={03},
  pages={097},
  year={2007},
  publisher={IOP Publishing}
}

@article{looijengalunts,
  title={A Lie algebra attached to a projective variety},
  author={Looijenga, Eduard and Lunts, Valery A},
  journal={Inventiones mathematicae},
  volume={129},
  number={2},
  pages={361--412},
  year={1997},
  publisher={Springer}
}

@article{levay,
  title={Two-center black holes, qubits, and elliptic curves},
  author={L{\'e}vay, P{\'e}ter},
  journal={Physical Review D},
  volume={84},
  number={2},
  pages={025023},
  year={2011},
  publisher={APS}
}

@article{dualityinvariantpartition,
  title={Duality-invariant partition functions and automorphic superpotentials for (2, 2) string compactifications},
  author={Ferrara, Sergio and Kounnas, Costas and Zwirner, Fabio and L{\"u}st, Dieter},
  journal={Nucl. Phys. B},
  volume={365},
  number={CERN-TH-6090-91},
  pages={431--466},
  year={1991}
}

@article{ferrarakallosh,
  title={Supersymmetry and attractors},
  author={Ferrara, Sergio and Kallosh, Renata},
  journal={Physical Review D},
  volume={54},
  number={2},
  pages={1514},
  year={1996},
  publisher={APS}
}

@misc{medolgachev,
      title={Attractors are not algebraic}, 
      author={Yeuk Hay Joshua Lam and Arnav Tripathy},
      year={2020},
      eprint={2009.12650},
      archivePrefix={arXiv},
      primaryClass={math.NT}
}

@article{stromingerspecial,
  title={Special geometry},
  author={Strominger, Andrew},
  journal={Communications in mathematical physics},
  volume={133},
  number={1},
  pages={163--180},
  year={1990},
  publisher={Springer}
}

@article{maniveltopics,
  title={Topics on the geometry of homogeneous spaces},
  author={Manivel, Laurent},
  journal={arXiv preprint arXiv:2001.11865},
  year={2020}
}

@article{duke,
  title={Hyperbolic distribution problems and half-integral weight Maass forms},
  author={Duke, William},
  journal={Inventiones mathematicae},
  volume={92},
  number={1},
  pages={73--90},
  year={1988},
  publisher={Springer}
}

\end{document}